\documentclass[conference]{IEEEtran}%
\usepackage{amssymb}
\usepackage{amsmath}
\usepackage{amsfonts}
\usepackage{hyperref}
\usepackage{graphicx}%
\providecommand{\U}[1]{\protect\rule{.1in}{.1in}}
\newtheorem{theorem}{Theorem}

\newtheorem{lemma}[theorem]{Lemma}

\newtheorem{proposition}[theorem]{Proposition}

\begin{document}

\title{Optimal control of anthracnose using mixed strategies}

%

\author{
\authorblockN{David Jaures FOTSA MBOGNE}
\authorblockA{
Department of Mathematics and Computer Science\\
ENSAI, The University of Ngaoundere\\
Email: mjdavidfotsa@gmail.com
}
\and\authorblockN{Christopher THRON}
\authorblockA{ \\
Texas A\&M University, Central Texas\\
Email:thron@ct.tamus.edu
}
}%
%

\maketitle
%

\begin{abstract}%
In this paper we propose and study a spatial diffusion model for the control
of anthracnose disease in a bounded domain. The model is a generalization of
the one previously developed in \cite{fotsa}. We use the model to simulate two
different types of control strategies against anthracnose disease. Strategies
that employ chemical fungicides are modeled using a continuous control
function; while strategies that rely on cultivational practices (such as
pruning and removal of mummified fruits) are modeled with a control function
which is discrete in time (though not in space). Under weak smoothness
conditions on parameters we demonstrate the well-posedness of the model by
verifying existence and uniqueness of the solution for given initial
conditions. We also show that the set $\left[  0,1\right]  $ is positively
invariant. We first study control by pulse strategy only, then analyze the
simultaneous use of continuous and pulse strategies. In each case we specify a
cost functional to be minimized, and we demonstrate the existence of optimal
control strategies that can be evaluated numerically using the gradient method
presented in \cite{anita}. We discuss the results of numerical simulations
both for a spatially-averaged version of the model and for the full model.



\end{abstract}%

\section{Introduction}

Anthracnose is a phytopathology which attacks several commercial tropical
crops such as coffee. The Anthracnose of coffee is known under the name coffee
berry disease (CBD) and its pathogen is the \textit{Colletotrichum}
\textit{kahawae}, an ascomycete fungus. The literature on Anthracnose
pathosystem is extensive
\cite{bieysse,boisson,chen,jeffries,mouen09,muller70,wharton}. There have been
several attempts to model the spread of CBD and to identify efficient control
strategies
\cite{danneberger,dodd,duthie,mouen07,mouen09,mouen072,mouen03,mouen08,wharton}%
. Possible control methods include genetic methods
\cite{bella,bieysse,boisson,ganesh,silva}, biological control \cite{durand},
chemical control \cite{boisson,muller67,muller71} and cultivational practices
\cite{boisson,mouen072,mouen03,mouen08,wharton}. Chemical methods appear to be
the most effective, but present ecological risks. Moreover, inadequate
application of chemical treatments can induce resistance in the pathogen
\cite{ramos}.

A dynamical spatial model of anthracnose infection that includes chemical
control was proposed and analysed in \cite{fotsa}; this paper also showed how
to optimize the use of the chemical control with respect to a given cost
functional. The disease dynamics were represented by an inhibition rate that
satisfies a reaction-diffusion partial differential equation with coefficients
that depend on space and time. The present paper adds to the above model the
possibility of a pulse control strategy that represents cultivational
practices such as pruning old, infected twigs and removing mummified fruits.
Such actions are commonly performed at discrete times at regular intervals. An
additional enhancement to the model results from our relaxing the regularity
conditions on the model parameters that were imposed in \cite{fotsa}. The
enhanced model is able to take into account the fact that in realistic
situations the application of antifungal compounds is typically not continuous
in time (although the action of these compounds once applied is continuous).

The remainder of the paper has the following structure. In section
\ref{System}, we present the system model and explain the significance of the  model parameters. In
section \ref{Wellposedness} we establish the well-posedness of the model
$\left(  \ref{ModelPulse1}\right)  -\left(  \ref{ModelPulse4}\right)  $ (under
certain conditions) and its spatially-averaged version. Section \ref{SectionCOV}
proves existence of an optimal control strategy based only on the pulse
strategy, for both the spatially-averaged model and the general model. Some properties of the optimal control strategy are proven, and an algorithm for finding the optimal pulse
strategy is given, which applies both spatially-averaged and general models.  Section
\ref{SectionCOUV}proves the existence of an  optimal control strategy using simultaneously pulse and continuous strategies, for both the spatially-averaged and the general model. Some extremal properties of these strategies are also established.
In Section~\ref{SectionModSim} we present system simulations of both the spatially-averaged and general models that demonstrate properties of the optimal pulse-only control. 
 Finally, in section \ref{Discussion} we summarize our conclusions.

\section{System model\label{System}}

The model of anthracnose infection discussed in \cite{fotsa} expressed the
disease dynamics in terms of an inhibition rate $\theta$ that satisfies the
following equations:%

\begin{align}
&  \partial_{t}\theta=\alpha\left(  t,x,\theta\right)  \left(  1-\theta
/\left(  1-\sigma u\left(  t,x\right)  \right)  \right)  +\operatorname{div}%
\left(  A\left(  t,x,\theta\right)  \nabla\theta\right) \nonumber\\
&  \qquad\quad\text{ on } (t,x) \in\mathbb{R}_{+}^{\ast}\times\Omega
;\label{ModelOrig1}\\
&  \left\langle A\left(  t,x,\theta\right)  \nabla\theta\left(  t,x\right)
,n\left(  x\right)  \right\rangle =0,\text{ on } \mathbb{R}_{+}^{\ast}%
\times\partial\Omega;\label{ModelOrig2}\\
&  \theta\left(  0,x\right)  =\rho\left(  x\right)  , x\in\overline{\Omega
}\subseteq\mathbb{R}^{3}, \label{ModelOrig3}%
\end{align}
where

$\alpha(t,x,\theta)$ is a positive real-valued function defined on $(t,x)
\in\mathbb{R}_{+}^{\ast}\times\Omega$;

$\sigma$ is a real parameter satisfying $0 \leq\sigma\leq1$;

$u(x,t)$ is a real-valued function with $0 \le u(x,t) \le1$, defined on the
same domain as $\alpha$;

$A(t,x,\theta)$ is a $3 \times3$ matrix function which is positive definite
for all $t,x,\theta$, defined on the same domain as $\alpha$;

$\Omega$ is an open, bounded subset of $\mathbb{R}^{3}$;

$\partial\Omega$ is the boundary of $\Omega$, which is assumed to satisfy
$\partial\Omega\in H^{1}\left(  \mathbb{R}^{2};\mathbb{R}\right)  $;

$n\left(  x\right)  $ denotes the normal vector to the boundary at $x
\in\partial\Omega$;

$\rho(x)$ is a real-valued function satisfying $0 \leq\rho(x) \leq1$ for $x
\in\overline{\Omega}$.

The various terms in the model equations $(\ref{ModelOrig1})
-(\ref{ModelOrig3})$ have practical interpretations as follows. (See reference
\cite{fotsa} for a more detailed description.) The function $\alpha$
represents the inhibition pressure, which depends on climatic and
environmental conditions \cite{danneberger,dodd,duthie}. It is appropriate to
model $\alpha$ as an almost-periodic function, taking into account yearly
seasonal changes. The term $\operatorname{div}\left(  A\nabla\theta\right)  $
accounts for the (possibly anisotropic) diffusive spatial spreading of
inhibition rate in the open domain $\Omega\subset\mathbb{R}^{3}$, where the
matrix $A$ contains the space- and time-dependent diffusion coefficients. The
boundary condition $\left\langle A\nabla\theta,n\right\rangle =0$ guarantees
that there is no net flux of inhibition rate between $\Omega$ and its
exterior. The function $u(x,t)$ expresses the influence of chemical control
(that is, the application of fungicides) on the inhibition rate. $1-\sigma$ is
the inhibition rate corresponding to epidermis penetration. Once the epidermis
has been penetrated, the inhibition rate cannot fall below this value, even
under maximum control effort.

In the current paper, we propose the following modified model, that includes
the possibility of impulsive control:
\begin{align}
&  \partial_{t}\theta=\alpha\left(  t,x\right)  \left(  1-\theta/\left(
1-\sigma u\left(  t,x\right)  \right)  \right)  +\operatorname{div}\left(
A\left(  t,x\right)  \nabla\theta\right)  ,\nonumber\\
&  \qquad\quad(t,x)\in\left(  \mathbb{R}_{+}^{\ast}\setminus\left\{  \tau
_{i}\right\}  _{i\in\mathbb{N}}\right)  \times\Omega;\label{ModelPulse1}\\
&  \theta\left(  \tau_{i}^{+},x\right)  =v_{i}\left(  x\right)  \theta\left(
\tau_{i},x\right)  ,\text{ }i\in\mathbb{N}^{\ast},x\in\Omega
;\label{ModelPulse2}\\
&  \left\langle A\left(  t,x\right)  \nabla\theta\left(  t,x\right)  ,n\left(
x\right)  \right\rangle =0,\nonumber\\
&  \qquad\quad(t,x)\in\left(  \mathbb{R}_{+}^{\ast}\setminus\left\{  \tau
_{i}\right\}  _{i\in\mathbb{N}}\right)  \times\partial\Omega
;\label{ModelPulse3}\\
&  \theta\left(  0,x\right)  =\rho\left(  x\right)  \geq0,x\in\overline
{\Omega}\subseteq\mathbb{R}^{3}, \label{ModelPulse4}%
\end{align}
where $\alpha$, $u$, $\sigma$, $A$, $\Omega$, $n\left(  x\right)  $, and
$\rho(x)$ are as above (except that $\alpha$ and $A$ no longer depend on
$\theta$) and

$\left(  t_{k}\right)  _{k\in\mathbb{N}}$ is an increasing sequence of
nonnegative reals such that $\underset{k\rightarrow\infty}{\lim}t_{k}=\infty$;

$\tau_{0} =t_{0}=0$;

$\tau_{i}=\inf\left\{  t_{k}>\tau_{i-1};k>0\text{ and }\left\Vert
\theta\left(  t_{k},.\right)  \right\Vert _{L^{2}\left(  \Omega\right)  }%
\geq\sigma^{\ast}\left\vert \Omega\right\vert \right\}  $, where $\left\vert
\Omega\right\vert $ denotes the volume $%
{\displaystyle\int\nolimits_{\Omega}}
dx$;

$\sigma^{\ast}\in\mathbb{R}^{+}$ is a threshold value such that the inhibition
rate is not measurable under it and is observable for values greater than
$\sigma^{\ast}$;

$v_{i}(x),i=1,2,3,\ldots$ is a $[0,1]$-valued function defined on $x\in\Omega$.

In $(\ref{ModelPulse1})-(\ref{ModelPulse4})$, the inhibition rate $\theta$ is
assumed to be left continuous with respect to time: $\theta\left(  t,.\right)
=\theta\left(  t^{-},x\right)  $, where $\theta\left(  t^{-},x\right)  $
denotes $\underset{s\rightarrow t,s<t}{\lim}\theta\left(  s,x\right)  $. (We
shall also use the notation $\theta\left(  t^{+},x\right)  $ to denote
$\underset{s\rightarrow t,s>t}{\lim}\theta\left(  s,x\right)  . $)

Cultivational practices are included in the model $(\ref{ModelPulse1}%
)-(\ref{ModelPulse4})$ as follows. The sequence $\left(  t_{k}\right)
_{k\in\mathbb{N}}$ represents times at which cultivational interventions are
possible. For example, in coffee cultivation it would be reasonable to take
the $t_{k}^{\prime}s$ as regularly spaced with an interval of one week. The
intervals between intervention times reflect the fact that continuous exertion
of cultivational interventions such as pruning is neither practical nor
efficient. At each potential intervention time $t_{k}$, intervention only
takes place if the infection is sufficiently serious, as determined by the
threshold condition $\left\Vert \theta(t_{k},.)\right\Vert _{H^{1}(\Omega
)}\geq\sigma^{\ast}$. The degree of intervention at time $\tau_{i}$ at each
point $x\in\Omega$ is given by $v_{i}(x)$: $v_{i}(x)=1$ corresponds to no
change in the inhibition rate at $x$, while $v_{i}(x)=0$ reduces the
inhibition rate $\theta(x)$ to $0$.

\section{Well-posedness of the model\label{Wellposedness}}

In this section we verify that the model $(\ref{ModelPulse1}%
)-(\ref{ModelPulse4})$ is well-posed (under certain conditions) and that the
inhibition rate $\theta(t,x)$ always remains between $0$ and $1$. We first
establish well-posedness of a spatially-averaged version of the
model, and then use similar techniques to prove well-posedness of the
spatially-dependent model.

\subsection{Well-posedness of spatially-averaged model\label{SubsectionIDEModel}}

A spatially-averaged version of the model $(\ref{ModelPulse1})-(\ref{ModelPulse4})$
may be obtained by taking spatial averages of the equations. The averaged
model is much simpler to work with than the general model; however, the same
tools used to prove well-posedness for the averaged model can be generalized
to apply to the general model.

Define
\begin{equation}
\Theta\left(  t\right)  \equiv\frac{1}{\left\vert \Omega\right\vert
}{\displaystyle\int\nolimits_{\Omega}}\theta\left(  t,x\right)  dx.
\end{equation}
If we suppose that $\alpha,$ $u$ and $v_{i}$ are functions of $t$ only and not
of $x$, then $\Theta(t)$ satisfies the following impulsive differential
equation:
\begin{align}
&  d_{t}\Theta=\alpha\left(  t\right)  \left(  1-\Theta/\left(  1-\sigma
u\left(  t\right)  \right)  \right)  ,\text{ on }\mathbb{R}_{+}^{\ast
}\setminus\left(  \tau_{i}\right)  _{i\in\mathbb{N}};\label{ModelPulse1IDE}\\
&  \Theta\left(  \tau_{i}^{+}\right)  =v_{i}\Theta\left(  \tau_{i}\right)
,~i\in\mathbb{N}^{\ast},\label{ModelPulse2IDE}\\
&  \Theta\left(  0\right)  =\Theta_{0}\in\left[  0,1\right]  ,
\label{ModelPulse3IDE}%
\end{align}
where
\begin{align}
&  \tau_{0}=t_{0}=0;\\
&  \tau_{i}=\inf\left\{  t_{k}>\tau_{i-1};k>0,\text{ }\Theta\left(  t_{k}%
\right)  \geq\sigma^{\ast}\right\}  . \label{ModelPulse2TIDE}%
\end{align}
Note that the divergence term in (\ref{ModelPulse1}) vanishes in the
averaged model, due to the no-flux boundary conditions (\ref{ModelPulse3}).

A solution of $\left(  \ref{ModelPulse1IDE}\right)  -\left(
\ref{ModelPulse3IDE}\right)  $ is a piecewise absolutely continuous
real-valued function that satisfies the following equation in each interval
$\left]  \tau_{k},\tau_{k+1}\right]  $:%
\begin{equation}
\Theta\left(  t\right)  =\Theta\left(  \tau_{k}^{+}\right)
+{\displaystyle\int\nolimits_{\tau_{k}}^{t}} \alpha\left(  s\right)  \left(
1-\Theta\left(  s\right)  /\left(  1-\sigma u\left(  s\right)  \right)
\right)  ds. \label{ThetaIntegral}%
\end{equation}
We impose the following conditions on $\alpha$ and $u$ to ensure the
solution's existence and uniqueness:

\textbf{(H1): }$~~\alpha\in L_{loc}^{\infty}\left(
\mathbb{R}
_{+};%
\mathbb{R}
_{+}\right)  $.

\textbf{(H2): }$~~u\in L^{\infty}\left(
\mathbb{R}
_{+};\left[  0,1\right]  \right)  $.

\begin{proposition}
\label{boundednessIDE}If $\Theta$ is a maximal solution of $\left(
\ref{ModelPulse1IDE}\right)  -\left(  \ref{ModelPulse3IDE}\right)  $, then
$\Theta$ is $\left[  0,1\right]  -$valued.
\end{proposition}

\begin{proof}
Let $\Theta$ be a solution of $\left(  \ref{ModelPulse1IDE}\right)  -\left(
\ref{ModelPulse3IDE}\right)  $. Since $v_{i}\in\left[  0,1\right]  $ it
suffices to establish that the restriction of $\Theta$ on $\left]  t_{0}%
,t_{1}\right]  $ is $\left[  0,1\right]  -$valued. Let $f=\max\left\{
0,-\Theta\right\}  $ and $g=\max\left\{  0,\Theta-1\right\}  $.

We first prove that $\Theta\geq0$. Let $U\subset\left]  t_{0},t_{1}\right]  $
be the set where $f$ is positive. Since $f$ is continuous, it follows that $U$
is open in $\left]  t_{0},t_{1}\right]  $. Suppose that $U$ is nonempty; then
$U$ is the disjoint union of open subintervals of $\left]  t_{0},t_{1}\right]
$. Let $U^{\prime}\neq\emptyset$ be one of these intervals. Then from
(\ref{ThetaIntegral}) and the definition of $f$, for $t\in U^{\prime}$ we
have
\begin{align*}
f\left(  t\right)   &  =%
{\displaystyle\int\nolimits_{\inf U^{\prime}}^{t}}
-\alpha\left(  s\right)  \left(  1+f\left(  s\right)  /\left(  1-\sigma
u\left(  s\right)  \right)  \right)  ds\\
&  \leq0.
\end{align*}
But we know $f(t) > 0$, since $t \in U^{\prime}$; this contradiction implies
that $U$ is empty. It follows that $f=0$ on $\left]  t_{0},t_{1}\right]  $,
which implies that $\Theta\ge0$ on $\left]  t_{0},t_{1}\right]  $.

To prove that $\Theta\leq1$ on $\left]  t_{0},t_{1}\right]  $, we may use an
almost identical argument, with $g(x)$ replacing $f(x)$.
\end{proof}

\begin{proposition}
\label{ExistenceGlobalSolutionIDE}The problem $\left(  \ref{ModelPulse1IDE}%
\right)  -\left(  \ref{ModelPulse3IDE}\right)  $ has a unique global solution.
\end{proposition}

\begin{proof}
For existence, it suffices to establish existence of a local solution and use
Proposition \ref{boundednessIDE} to conclude the result based on Theorem 5.7
of \cite{benzoni}. It also suffices to restrict ourselves to the set $\left[
t_{0}^{+},t_{1}\right]  $. The function $\left(  t,x\right)  \in\left[
t_{0}^{+},t_{1}\right]  \times\left[  0,1\right]  \mapsto\alpha\left(
t\right)  \left(  1-x/\left(  1-\sigma u\left(  t\right)  \right)  \right)  $
is integrable with respect to $t$, Lipschitz continuous with respect to $x$,
and upper bounded by $\alpha$ which is also integrable with respect to $t$.
Then by the Carath\'{e}odory theorem (see \cite{coddington}) there is a local
$\left[  0,1\right]  -$valued solution.

We now prove uniqueness of the solution. Given that $x$ and $y$ are solutions
on $\left[  t_{0}^{+},t_{1}\right]  $, we have that
\begin{align*}
\left\vert x\left(  t\right)  -y\left(  t\right)  \right\vert  &
\leq\left\vert x\left(  t_{0}^{+}\right)  -y\left(  t_{0}^{+}\right)
\right\vert \\
&  +%
{\displaystyle\int\nolimits_{t_{0}}^{t}}
\alpha\left(  s\right)  \left\vert y\left(  s\right)  -x\left(  s\right)
\right\vert /\left(  1-\sigma u\left(  s\right)  \right)  ds.
\end{align*}
Using the Gronwall lemma, we get
\begin{align*}
\left\vert x\left(  t\right)  -y\left(  t\right)  \right\vert  &
\leq\left\vert x\left(  t_{0}^{+}\right)  -y\left(  t_{0}^{+}\right)
\right\vert \\
&  \times\exp\left(
{\displaystyle\int\nolimits_{t_{0}}^{t}}
\alpha\left(  s\right)  /\left(  1-\sigma u\left(  s\right)  \right)
ds\right)
\end{align*}
and more generally if $t\in$ $\left[  t_{k}^{+},t_{k+1}\right]  $
\begin{align}
\left\vert x\left(  t\right)  -y\left(  t\right)  \right\vert  &  \leq\left(
{\displaystyle\prod\nolimits_{i=0}^{k}}
v_{i}\right)  \left\vert x\left(  t_{0}\right)  -y\left(  t_{0}\right)
\right\vert \nonumber\\
&  \times\exp\left(
{\displaystyle\int\nolimits_{t_{0}}^{t}}
\alpha\left(  s\right)  /\left(  1-\sigma u\left(  s\right)  \right)
ds\right)  \label{contin}%
\end{align}
It follows that the solution is unique and depends continuously on initial conditions.
\end{proof}

It is important to notice that (\ref{contin}) implies that the solution of
$\left(  \ref{ModelPulse1IDE}\right)  -\left(  \ref{ModelPulse3IDE}\right)  $
is continuous with respect to control strategies $u$ and $\left(
v_{i}\right)  $.

\subsection{Well-posedness of the general model}

A solution $\theta$ of $\left(  \ref{ModelPulse1}\right)  -\left(
\ref{ModelPulse4}\right)  $ is a piecewise absolutely continuous function of
time, such that $\forall t\geq0$ the function $\theta\left(  t,.\right)  \in
H^{2}(\Omega;\mathbb{R})$ and satisfies
\begin{align}
\theta\left(  t,x\right)   &  =\theta\left(  \tau_{k}^{+},x\right) \nonumber\\
&  +{\displaystyle\int\nolimits_{\tau_{k}}^{t}}\alpha\left(  s,x\right)
\left(  1-\theta\left(  s,x\right)  /\left(  1-\sigma u\left(  s,x\right)
\right)  \right)  ds\nonumber\\
&  +{\displaystyle\int\nolimits_{\tau_{k}}^{t}}\operatorname{div}\left(
A\left(  s,x\right)  \nabla\theta\left(  s,x\right)  \right)  ds;
\label{weak_1}%
\end{align}%
\begin{align}
&  \left\langle A\left(  t,x\right)  \nabla\theta\left(  t,x\right)  ,n\left(
x\right)  \right\rangle =0\nonumber\\
&  \qquad\qquad\qquad\qquad\qquad\text{ on }\left(  \mathbb{R}_{+}^{\ast
}\setminus\left(  \tau_{i}\right)  _{i\in\mathbb{N}}\right)  \times
\partial\Omega; \label{weak2}%
\end{align}
for all $t\in]\tau_{k},\tau_{k+1}]$ and $\forall x\in\Omega$. We also have
\begin{equation}
\theta\left(  0,x\right)  =\rho\left(  x\right)  \in\left[  0,1\right]
,\text{ }\forall x\in\overline{\Omega}\subseteq\mathbb{R}^{3}. \label{weak_3}%
\end{equation}

We further define a \textit{weak solution} $\theta$ of $\left(
\ref{ModelPulse1}\right)  -\left(  \ref{ModelPulse4}\right)  $ to be a
piecewise absolutely continuous function with respect to time which satisfies
(\ref{weak_3}) and $\forall t\geq0,$ the function $\theta\left(  t,.\right)
\in H^{1}(\Omega;\mathbb{R})$ satisfies (\ref{weak2}) and the following
\textquotedblleft weak\textquotedblright\ form of (\ref{weak_1}):
\begin{align*}
&
{\displaystyle\int\nolimits_{\Omega}}
\theta\left(  t,x\right)  \psi\left(  t,x\right)  dx\\
&  =%
{\displaystyle\int\nolimits_{\Omega}}
\theta\left(  \tau_{k}^{+},x\right)  \psi\left(  \tau_{k}^{+},x\right)  dx\\
&  -%
{\displaystyle\int\nolimits_{\Omega}}
{\displaystyle\int\nolimits_{\tau_{k}}^{t}}
\left\langle A\left(  s,x\right)  \nabla\theta\left(  s,x\right)  ,\nabla
\psi\left(  s,x\right)  \right\rangle ds\\
&  +%
{\displaystyle\int\nolimits_{\Omega}}
{\displaystyle\int\nolimits_{\tau_{k}}^{t}}
\alpha\left(  s,x\right)  \psi\left(  s,x\right)  \left(  1-\theta\left(
s,x\right)  /\left(  1-\sigma u\left(  s,x\right)  \right)  \right)  ds\,dx,
\end{align*}
where $t\in]\tau_{k},\tau_{k+1}]$ and $\psi\in H^{1}\left(  \Omega;%
\mathbb{R}
\right)  $. We make the following additional assumptions in order to guarantee
existence and uniqueness of the weak solution:

\textbf{(H3): }$\alpha\in L_{loc}^{\infty}\left(
\mathbb{R}
_{+};L^{\infty}\left(  \Omega;%
\mathbb{R}
_{+}\right)  \right)  $;

\textbf{(H4): }$\forall i,j\in\left\{  1,2,3\right\}  ,$ $a_{ij}\in
L_{loc}^{\infty}\left(
\mathbb{R}
_{+};W^{1,\infty}\left(  \Omega;%
\mathbb{R}
\right)  \right)  $;

\textbf{(H5): }$\exists\delta\in%
\mathbb{R}
_{+}^{\ast}$ such that $\forall t\in%
\mathbb{R}
_{+},\forall w\in H^{1}\left(  \Omega;%
\mathbb{R}
\right)  ,$
\[%
{\displaystyle\int\nolimits_{\Omega}}
\left\langle A\left(  t,x\right)  \nabla w\left(  x\right)  ,\nabla w\left(
x\right)  \right\rangle dx\geq\delta%
{\displaystyle\int\nolimits_{\Omega}}
\left\langle \nabla w\left(  x\right)  ,\nabla w\left(  x\right)
\right\rangle dx;
\]

\textbf{(H6): }$u\in L^{\infty}\left(
\mathbb{R}
_{+};L^{\infty}\left(  \Omega;\left[  0,1\right]  \right)  \right)  $;

\textbf{(H7): }$\forall i\in%
\mathbb{N}
,v_{i}\in L^{\infty}\left(  \Omega;\left[  0,1\right]  \right)  $.

As preliminary to proving existence and uniqueness of weak solutions, we first
establish the boundedness of solutions. The proof is similar to that of
Proposition~\ref{boundednessIDE}.

\begin{proposition}
\label{boundednessIPDE}If $\theta$ is a maximal solution of $\left(
\ref{ModelPulse1}\right)  -\left(  \ref{ModelPulse4}\right)  $ then $\theta$
is $\left[  0,1\right]  $-valued.
\end{proposition}

\begin{proof}
Let $\theta$ be a solution of $\left(  \ref{ModelPulse1}\right)  -\left(
\ref{ModelPulse4}\right)  $. Since $v_{i}\in L^{\infty}\left(  \Omega;\left[
0,1\right]  \right)  $ it suffices to establish the result for the restriction
of $\theta$ on $\left]  t_{0},t_{1}\right]  $. Let $f=\max\left\{
0,-\theta\right\}  $ and $g=\max\left\{  0,\theta-1\right\}  $.

We first prove that $\theta\geq0$. Let $\mathcal{U} \subset\left]  t_{0}%
,t_{1}\right]  \times\Omega$ be the set where $f$ is positive. Let $\Omega_{t}
= (\{t\}\times\Omega) \cap\mathcal{U}$, and let $U = \{ t \,|\,\Omega_{t}
\neq\emptyset\}$. As in the proof of Proposition~\ref{boundednessIDE}, we may
choose an open subinterval $U^{\prime}\subset U$ which is a connected
component of $U$. For almost every time $t\in U^{\prime}$, the function
$f\left(  t,.\right)  \in H^{1}\left(  \Omega,\mathbb{R} \right)  $ and
\begin{align*}
&  \frac{1}{2}\partial_{t}\left\Vert f\left(  t,.\right)  \right\Vert
_{L^{2}\left(  \Omega;%
\mathbb{R}
\right)  }^{2}\\
&  =\frac{1}{2}\partial_{t}%
{\displaystyle\int\nolimits_{\Omega_{2}}}
f^{2}\left(  t,x\right)  dx\\
&  =%
{\displaystyle\int\nolimits_{\Omega}}
f\left(  t,x\right)  \partial_{t}f\left(  t,x\right)  dx\\
&  =-%
{\displaystyle\int\nolimits_{\Omega}}
\alpha\left(  t,x\right)  f\left(  t,x\right)  dx\\
&  +%
{\displaystyle\int\nolimits_{\Omega}}
\operatorname{div}\left(  A\left(  t,x\right)  \nabla f\left(  t,x\right)
\right)  f\left(  t,x\right)  dx\\
&  -%
{\displaystyle\int\nolimits_{\Omega}}
\alpha\left(  t,x\right)  f^{2}\left(  t,x\right)  /\left(  1-\sigma u\left(
t,x\right)  \right)  dx\\
&  =-%
{\displaystyle\int\nolimits_{\Omega}}
\alpha\left(  t,x\right)  f\left(  t,x\right)  dx\\
&  -%
{\displaystyle\int\nolimits_{\Omega}}
\left\langle A\left(  t,x\right)  \nabla f\left(  t,x\right)  ,\nabla f\left(
t,x\right)  \right\rangle dx\\
&  -%
{\displaystyle\int\nolimits_{\Omega}}
\alpha\left(  t,x\right)  f^{2}\left(  t,x\right)  /\left(  1-\sigma u\left(
t,x\right)  \right)  dx\\
&  \leq0.
\end{align*}
Since $\left\Vert f(\inf(U^{\prime}),.) \right\Vert _{L^{2}(\Omega
;\mathbb{R})} = 0$, it follows that $\left\Vert f(t,.) \right\Vert
_{L^{2}(\Omega;\mathbb{R})} = 0$ for all $t \in U^{\prime}$, which implies $f
\equiv0$ in $U^{\prime}\times\Omega$. It follows immediately that $f \equiv0$
in $U \times\Omega$, so that $\theta\geq0$ on $\left]  t_{0},t_{1}\right]  $.

A similar computation to the above can be used to show that $g \equiv0$ in $U
\times\Omega$, from which $\theta\le1$ follows immediately.

\end{proof}

From assumptions \textbf{(H4)-(H5)}, it follows that the following problem has
a unique solution in $H^{1}\left(  \Omega\right)  $ for an arbitrary but fixed
time $t>0$.%
\[
\left\{
\begin{array}
[c]{l}%
\operatorname{div}\left(  A\left(  t,x\right)  \nabla w\left(  x\right)
\right)  =f\left(  x\right)  ,\text{ }\forall x\in\Omega\\
\left\langle A\left(  t,x\right)  \nabla w\left(  x\right)  ,n\left(
x\right)  \right\rangle =0,\text{ }\forall x\in\partial\Omega
\end{array}
\right.
\]
where $f\in L^{2}\left(  \Omega\right)  $. Theorems 3.6.1 and 3.6.2 of
\cite{barbu} imply that there is a complete orthonormal system $\left\{
\varphi_{n}\left(  t,.\right)  \right\}  _{n\in%
\mathbb{N}
}\subset L^{2}\left(  \Omega\right)  $ of eigenfunctions and eigenvalues
$\left\{  \lambda_{n}\left(  t\right)  \right\}  $ such that $\forall n\in%
\mathbb{N}
,$
\[
\left\{
\begin{array}
[c]{l}%
\operatorname{div}\left(  A\left(  t,x\right)  \nabla\varphi_{n}\left(
t,x\right)  \right)  =\lambda_{n}\left(  t\right)  \varphi_{n}\left(
t,x\right)  ,\text{ }\forall x\in\Omega\\
\left\langle A\left(  t,x\right)  \nabla\varphi_{n}\left(  t,x\right)
,n\left(  x\right)  \right\rangle =0,\text{ }\forall x\in\partial\Omega
\end{array}
\right.
\]
Moreover, the sequence $\left(  \varphi_{n}\left(  t,.\right)  \right)  _{n\in%
\mathbb{N}
}$ is $H^{1}\left(  \Omega\right)  -$valued, and if $\partial\Omega$ is of
class $C^{2}$ then $\left(  \varphi_{n}\left(  t,.\right)  \right)  _{n\in%
\mathbb{N}
}$ is $H^{2}\left(  \Omega\right)  -$valued.

Now we make the following assumption

\textbf{(H8):} The sequence $\left(  \varphi_{n}\right)  $ is independent of
time (that is, $\varphi_{n}\left(  t,.\right)  =\varphi_{n}\left(  .\right)
,$ $\forall t>0$).

In particular, \textbf{(H8)} holds if $A\left(  t,.\right)  $ has the form
$\mu\left(  t\right)  B\left(  .\right)  $ with $\mu\left(  t\right)  \in%
\mathbb{R}
$, $\forall t\geq0$. In the special case where $A\left(  t,.\right)
=\mu\left(  t\right)  I$ ($I$ is the $3\times3$ identity matrix), we have
$\operatorname{div}\left(  A\left(  t,.\right)  \nabla w\right)  = \mu(t)
\Delta w$.

When \textbf{(H8)} is satisfied, a weak solution $\theta$ of $\left(
\ref{ModelPulse1IDE}\right)  -\left(  \ref{ModelPulse3IDE}\right)  $ can be
written as $%
{\displaystyle\sum\nolimits_{n=0}^{\infty}}
\theta_{n}\varphi_{n}$, where each $\theta_{n}$ is an absolutely continuous
function of time that for $t\in\left]  \tau_{k},\tau_{k+1}\right]  ,$
satisfies%
\begin{align}
&  \theta_{n}\left(  t\right)  =%
{\displaystyle\sum\nolimits_{m=0}^{\infty}}
\theta_{m}%
{\displaystyle\int\nolimits_{\Omega}}
\varphi_{m}\left(  x\right)  \varphi_{n}\left(  x\right)
dx\label{DecomposedSolution}\\
&  =\theta_{n}\left(  \tau_{k}^{+}\right)  +%
{\displaystyle\int\nolimits_{\tau_{k}}^{t}}
\lambda_{n}\left(  s\right)  \theta_{n}\left(  s\right)  ds\nonumber\\
&  -%
{\displaystyle\int\nolimits_{\tau_{k}}^{t}}
{\displaystyle\int\nolimits_{\Omega}}
\alpha\left(  s,x\right)  \varphi_{n}\left(  x\right)  /\left(  1-\sigma
u\left(  s,x\right)  \right) \nonumber\\
&  \times\sum\nolimits_{m=0}^{\infty}\theta_{m}\left(  s\right)  \varphi
_{m}\left(  x\right)  dxds\nonumber\\
&  +%
{\displaystyle\int\nolimits_{\tau_{k}}^{t}}
{\displaystyle\int\nolimits_{\Omega}}
\alpha\left(  s,x\right)  \varphi_{n}\left(  x\right)  dsdx.\nonumber
\end{align}
We are now ready to prove existence and uniqueness of weak solutions, given
the additional conditions we have imposed. The proof parallels that of
Proposition~\ref{ExistenceGlobalSolutionIDE}.

\begin{theorem}
\label{ExistenceGlobalSolutionIPDE}Under conditions (\textbf{H3}%
)--(\textbf{H8}), the problem $\left(  \ref{ModelPulse1}\right)  -\left(
\ref{ModelPulse4}\right)  $ has a unique global weak solution $\theta$.
Moreover, if $\partial\Omega$ is of class $C^{2}$ then $\forall t>0,$
$\theta\left(  t,.\right)  \in H^{2}\left(  \Omega\right)  $.
\end{theorem}

\begin{proof}
For existence, it suffices to establish existence of a local solution and use
Proposition \ref{boundednessIPDE} to conclude the result. It also suffices to
restrict ourselves to the set $\left[  t_{0}^{+},t_{1}\right]  $. Let $\ell
^{2}(\mathbb{R})$ denote the Hilbert space of real-valued sequences $\left\{
s_{n}\right\}  _{n\in%
\mathbb{N}
}$ such that $\sum\nolimits_{n\in%
\mathbb{N}
}\left\vert s_{n}\right\vert ^{2}<\infty$ . Consider the operator
$\mathcal{G}:D\left(  \mathcal{G}\right)  \rightarrow\ell^{2}\left(
\mathbb{R} \right)  $ where $D\left(  \mathcal{G}\right)  \subseteq\left[
t_{0}^{+},t_{1}\right]  \times\ell^{2}\left(  \mathbb{R} \right)  $ defined
by
\begin{align*}
&  \mathcal{G}\left(  t,y\right) \\
&  = \left(  \lambda_{n}\left(  t\right)  y_{n}-\sum\nolimits_{m=1}^{\infty
}y_{m}{\displaystyle\int\nolimits_{\Omega}} \frac{\alpha\left(  t,x\right)
\varphi_{m}\left(  x\right)  \varphi_{n}\left(  x\right)  }{\left(  1-\sigma
u\left(  t,x\right)  \right)  }dx\right)  _{n\in\mathbb{N} }.
\end{align*}
The set $\bigcap\nolimits_{t\in\left[  t_{0}^{+},t_{1}\right]  }D\left(
\mathcal{G}\left(  t,.\right)  \right)  $ is a nonempty dense subset of
$\ell^{2}\left(
\mathbb{R}
\right)  , $ since it contains the set of stationary sequences converging to
$0$ which is also dense in $\ell^{2}\left(
\mathbb{R}
\right)  $. $\mathcal{G}$ is the infinitesimal generator of the evolution
system defined from $\left[  t_{0}^{+},t_{1}\right]  \times\left[  t_{0}%
^{+},t_{1}\right]  \times\bigcap\nolimits_{t\in\left[  t_{0}^{+},t_{1}\right]
}D\left(  \mathcal{G}\left(  t,.\right)  \right)  $ to $\ell^{2}\left(
\mathbb{R}
\right)  $ by $(\forall n\in%
\mathbb{N}
$, $\forall t_{0}^{+}\leq s\leq t\leq t_{1}$)
\begin{align*}
\left(  U\left(  s,t\right)  y\right)  _{n}  &  =%
{\displaystyle\int\nolimits_{\Omega}}
\exp\left(  \int\nolimits_{s}^{t}\left(  \lambda_{n}\left(  \tau\right)
-\frac{\alpha\left(  \tau,x\right)  }{\left(  1-\sigma u\left(  \tau,x\right)
\right)  }\right)  d\tau\right) \\
&  \times\sum\nolimits_{m=1}^{\infty}y_{m}\varphi_{m}\left(  x\right)
\varphi_{n}\left(  x\right)  dx
\end{align*}
Hence, if we set $\forall n\in%
\mathbb{N}
$, $\theta_{n}\left(  t_{0}^{+}\right)  =y_{n}$ then the solution of $\left(
\ref{ModelPulse1}\right)  -\left(  \ref{ModelPulse4}\right)  $ is given by
\begin{align*}
\theta\left(  t,x\right)   &  =\sum\nolimits_{n=0}^{\infty}\varphi_{n}\left(
x\right)  \left(  U\left(  t_{0}^{+},t\right)  y\right)  _{n}\\
&  +\int\nolimits_{t_{0}}^{t}\varphi_{n}\left(  x\right)  \left(  U\left(
s,t\right)  z\left(  s\right)  \right)  _{n}%
\end{align*}
where $z_{n}\left(  s\right)  =%
{\displaystyle\int\nolimits_{\Omega}}
\alpha\left(  s,x\right)  \varphi_{n}\left(  x\right)  dx$.
\end{proof}

Just as with the averaged model, the solution of $\left(  \ref{ModelPulse1}%
\right)  -\left(  \ref{ModelPulse4}\right)  $ is continuous with respect to
the controls $u$ and $\left(  v_{i}\right)  $.

\section{Optimal control with purely impulsive stategy\label{SectionCOV}}

In this section we consider the behavior of controlled solutions on a fixed
time interval $\left[  0,T\right]  , $ where $T\in\left[  \tau_{k},\tau
_{k+1}\right[  $. Practically, this interval may be taken as representing one
annual production period. As in previous sections, we first consider the
averaged model, and then use analogous methods to obtain results for the
main model. Although we are optimizing only with respect to the impulsive strategy, 
 we retain the continuous control function $u$ as a fixed function. In Section~\ref{SectionCOUV}, we will optimize 
also with respect to $u$.

\subsection{Averaged model with impulsive strategy\label{COVAgragatedModel}}

The aim of this subsection is to characterize the impulsive strategy $v^{\ast
}=\left(  v_{i}^{\ast}\right)  _{0\leq i\leq k}$ which minimizes the following
cost functional for the averaged model:%
\begin{equation}
J\left(  v\right)  =%
{\displaystyle\int\nolimits_{0}^{T}}
\Theta\left(  s\right)  ds+%
{\displaystyle\sum\nolimits_{i=0}^{k}}
c_{i}\left(  1-v_{i}\right)  \Theta\left(  \tau_{i}\right)  +C_{f}%
\Theta\left(  T\right)  , \label{eq:Jdef}%
\end{equation}
where $c=\left(  c_{i}\right)  _{i\in%
\mathbb{N}
}$ is a $%
\mathbb{R}
_{+}^{\ast}-$valued sequence of cost ratios related to the use of impulsive
control and $C_{f}$ is also a cost related to the final inhibition rate.
The existence of such an optimal strategy is guaranteed by the following proposition.

\begin{proposition}\label{prop:exist}
There is an optimal strategy $v^{\ast}=\left(  v_{i}^{\ast}\right)  _{0 \le i
\le k}$ which minimizes $J$.
\end{proposition}

\begin{proof}
Note that $0\leq J\leq1+T+ {\displaystyle\sum\nolimits_{i=0}^{k}} c_{i}$, so
it is possible to define $J^{\ast}\equiv\underset{v\in\left[  0,1\right]
^{k+1}}{\inf}J\left(  v\right)  $, and there is a sequence $\left\{
v^{n}\right\}  _{n\in%
\mathbb{N}
}$ such that the sequence $\left\{  J\left(  v^{n}\right)  \right\}  _{n\in%
\mathbb{N}
}$ converges to $J^{\ast}$. Since $\left[  0,1\right]  ^{k+1}$ is compact and
$J$ is continuous, there is a subsequence $\left\{  v^{n_{m}}\right\}  $ which
converges to $v^{\ast}=\left(  v_{i}^{\ast}\right)  _{0\leq i\leq k}\in\left[
0,1\right]  ^{k+1}$ such that $J\left(  v^{\ast}\right)  =J^{\ast}$.
\end{proof}

In the remainder of this subsection we characterize the optimal control
strategy in such a way that it may be computed. We have mentioned above that
the solution of $\left(  \ref{ModelPulse1IDE}\right)  -\left(
\ref{ModelPulse3IDE}\right)  $ is continuous with respect to control
strategies $u$ and $\left(  v_{i}\right)  $. Now we add the assumption that
the solution of $\left(  \ref{ModelPulse1IDE}\right)  -\left(
\ref{ModelPulse3IDE}\right)  $ is G\^{a}teaux differentiable with respect to
$v$.

Let $\Theta_{\underline{v}}$ be the solution of $\left(  \ref{ModelPulse1IDE}%
\right)  -\left(  \ref{ModelPulse3IDE}\right)  $ associated with a chosen
control strategy $\underline{v}$, and let $z_{v}$ be its directional
derivative, 
\begin{equation*}
z_{v}=D_{v}\Theta_{\underline{v}}
\equiv \lim_{\lambda \rightarrow 0}\dfrac{\Theta_{\underline{v}+\lambda v} - \Theta_{\underline{v}}}{\lambda}.
\end{equation*}
Then $z_{v}$ may be computed as follows. Let $\Phi$ be the semiflow
corresponding to $\left(  \ref{ModelPulse1IDE}\right)  -\left(
\ref{ModelPulse3IDE}\right)  $. $\Phi$ satisfies $\forall t\in\left]  \tau
_{i},\tau_{i+1}\right]  ,$%
\begin{align}
&  \Phi\left(  t,\tau_{i},x\right)  =x\exp\left(  -%
{\displaystyle\int\nolimits_{\tau_{i}}^{t}}
\alpha\left(  s\right)  /\left(  1-\sigma u\left(  s\right)  \right)
ds\right) \nonumber\\
&  \quad+%
{\displaystyle\int\nolimits_{\tau_{i}}^{t}}
\alpha\left(  s\right)  \exp\left(  -%
{\displaystyle\int\nolimits_{s}^{t}}
\alpha\left(  \tau\right)  /\left(  1-\sigma u\left(  \tau\right)  \right)
d\tau\right)  ds \label{eq:Phi}%
\end{align}
and
\[
\partial_{x}\Phi\left(  t,\tau_{i},x\right)  =\exp\left(  -%
{\displaystyle\int\nolimits_{\tau_{i}}^{t}}
\alpha\left(  s\right)  /\left(  1-\sigma u\left(  s\right)  \right)
ds\right)
\]
On the other hand
\begin{align*}
\Theta_{\underline{v}}\left(  t\right)   &  =\Phi\left(  t,\tau_{i}%
,\Theta_{\underline{v}}\left(  \tau_{i}^{+}\right)  \right) \\
&  =\Phi\left(  t,\tau_{i},\underline{v}_{i}\Theta_{\underline{v}}\left(
\tau_{i}\right)  \right) \\
&  =\Phi\left(  t,\tau_{i},\underline{v}_{i}\Phi\left(  \tau_{i},\tau
_{i-1},\underline{v}_{i-1}\Theta_{\underline{v}}\left(  \tau_{i-1}\right)
\right)  \right)
\end{align*}
We have $\forall t\in\left]  0,\tau_{1}\right]  $,
\[
z_{v}\left(  t\right)  =v_{0}\Theta_{\underline{v}}\left(  0\right)
\exp\left(  -%
{\displaystyle\int\nolimits_{0}^{t}}
\alpha\left(  s\right)  /\left(  1-\sigma u\left(  s\right)  \right)
ds\right)
\]
and $\forall t\in\left]  \tau_{i},\tau_{i+1}\right]  ,i\in%
\mathbb{N}
^{\ast}$ we get
\begin{align}
z_{v}\left(  t\right)   &  =\left(  v_{i}\Theta_{\underline{v}}\left(
\tau_{i}\right)  +\underline{v}_{i}z_{v}\left(  \tau_{i}\right)  \right)
\nonumber\\
&  \qquad\times\exp\left(  -%
{\displaystyle\int\nolimits_{\tau_{i}}^{t}}
\alpha\left(  s\right)  /\left(  1-\sigma u\left(  s\right)  \right)
ds\right) \nonumber\\
&  =v_{i}\Theta_{\underline{v}}(\tau_{i})\nonumber\\
&  \qquad\times\exp\left(  -{\displaystyle\int\nolimits_{\tau_{i}}^{t}}%
\alpha\left(  s\right)  /\left(  1-\sigma u\left(  s\right)  \right)
ds\right) \nonumber\\
&  \quad+%
{\displaystyle\sum\nolimits_{j=0}^{i-1}}
v_{j}\left(
{\displaystyle\prod\nolimits_{l=j}^{i}}
\underline{v}_{l}\right)  \Theta_{\underline{v}}\left(  \tau_{j}\right)
\nonumber\\
&  \qquad\times\exp\left(  -%
{\displaystyle\int\nolimits_{\tau_{j}}^{t}}
\alpha\left(  s\right)  /\left(  1-\sigma u\left(  s\right)  \right)
ds\right)  . \label{eq:zv_computation}%
\end{align}
It follows that $z_{v}(0^{+})=v_{0}\Theta_{0},$ and $\forall t\in
\mathbb{R}_{+}^{\ast}\setminus\left(  \tau_{i}\right)  _{i\in\mathbb{N}}$, we
have
\begin{equation}
dz_{v}/dt=-\alpha(t)z_{v}(t)/(1-\sigma u(t)). \label{eq:deriv_z}%
\end{equation}
Also, $\forall i\in%
\mathbb{N}
^{\ast}$ we have%
\begin{align}
z_{v}\left(  \tau_{i}^{+}\right)  =  &  v_{i}\Theta_{\underline{v}}(\tau_{i})+%
{\displaystyle\sum\nolimits_{j=0}^{i-1}}
v_{j}\left(
{\displaystyle\prod\nolimits_{l=j+1}^{i}}
\underline{v}_{l}\right)  \Theta_{\underline{v}}\left(  \tau_{j}\right)
\nonumber\\
&  \quad\times\exp\left(  -%
{\displaystyle\int\nolimits_{\tau_{j}}^{\tau_{i}}}
\alpha\left(  s\right)  /\left(  1-\sigma u\left(  s\right)  \right)
ds\right)  . \label{eq:zv_tau_averaged}%
\end{align}
Note that $z_{v}$ can be written in the form $z_{v}=\sum\Psi_{\underline{v}%
}^{j}v_{j}$, where the coefficients $\Psi_{\underline{v}}^{j}$ may be inferred
from the above expressions for $z_{v}$.

Let $J_v(\underline{v}) \equiv D_{v}J(\underline{v})$ be the directional derivative of $J$ in the direction of $v$
 for a chosen control strategy $\underline{v}$. From (\ref{eq:Jdef}) we
may compute
\begin{align}
J_v(\underline{v})=  &  C_{f}z_{v}\left(  T\right)  +{\displaystyle\int\nolimits_{0}^{T}%
}z_{v}\left(  s\right)  ds\nonumber\\
&  +{\displaystyle\sum\nolimits_{i=0}^{k}}c_{i}\left(  \left(  1-\underline{v}_{i}\right)  z_{v}\left(  \tau_{i}\right)  -v_{i}\Theta_{\underline{v}}\left(
\tau_{i}\right)  \right)  \label{eq:Jv}%
\end{align}
Then we have


\begin{proposition}
\label{eq:OptStrategyAverage}For any control strategy $\underline{v}$ we have
\begin{equation}%
J_v(\underline{v}) = 
{\displaystyle\sum\nolimits_{i=0}^{k}}
 \left(  p_{\underline{v}}\left(  \tau_{i}^{+}\right)  -c_{i}\right)  v_{i}%
\Theta_{\underline{v}}\left(  \tau_{i}\right),
\label{eq:CondOptStrategy0}%
\end{equation}
where $v \in V$,
\[
V \equiv \left\{  v\in%
\mathbb{R}
^{k+1};\exists\varepsilon>0;\underline{v}+\varepsilon v\in\left[  0,1\right]
^{k+1}\right\}  ,
\]
and $p_{\underline{v}}$ is solution of the following adjoint problem:
\begin{align}
d_{t}p_{\underline{v}}  &  =\alpha\left(  t\right)  p_{\underline{v}}/\left(  1-\sigma
u\left(  t\right)  \right)  -1,\text{ }\label{EqAdjp1}\\
&  \qquad t\in\left]  0,T\right]  \setminus\left\{  \tau_{i}\right\}  ,\text{
}i\in\left[  0,k\right]  \cap%
\mathbb{N}
^{\ast}\text{ };\nonumber\\
\text{ }p_{\underline{v}}\left(  T\right)   &  =C_f,\text{ }p_{\underline{v}}\left(
\tau_{i}\right)  =c_{i}\left(  1-\underline{v}_{i}\right)  +p_{\underline{v}}\left(
\tau_{i}^{+}\right)  \underline{v}_{i}. \label{EqAdjp2}%
\end{align}
In particular, in the case where $v^{\ast}$ is an optimal solution then  
\begin{equation}%
 J_{v}(v^{\ast}) 
= 
{\displaystyle\sum\nolimits_{i=0}^{k}}
 \left(  p_{v^{\ast}}\left(  \tau_{i}^{+}\right)  -c_{i}\right)  v_{i}%
\Theta_{v^{\ast}}\left(  \tau_{i}\right)
\ge 0.
\label{eq:CondOptStrategy}%
\end{equation}

\end{proposition}

\begin{proof}
An argument similar to that in the proof of Proposition
\ref{ExistenceGlobalSolutionIDE} shows that there is a unique absolutely
continuous solution $p_{\underline{v}}$ to the problem $\left(  \ref{EqAdjp1}%
\right)  -\left(  \ref{EqAdjp2}\right)  $ which satisfies $\forall t\in\left]
\tau_{i},\tau_{i+1}\right]  ,$%
\begin{align*}
p_{\underline{v}}\left(  t\right)   &  =\left(  c_{i+1}\left(  1-\underline{v}_{i+1}\right)  +p_{\underline{v}}\left(  \tau_{i+1}^{+}\right)  \underline{v}_{i+1}\right) \\
&  \times\exp\left(  -%
{\displaystyle\int\nolimits_{t}^{\tau_{i+1}}}
\alpha\left(  s\right)  /\left(  1-\sigma u\left(  s\right)  \right)
ds\right) \\
&  +%
{\displaystyle\int\nolimits_{t}^{\tau_{i+1}}}
\exp\left(  -%
{\displaystyle\int\nolimits_{t}^{s}}
\alpha\left(  \tau\right)  /\left(  1-\sigma u\left(  \tau\right)  \right)
d\tau\right)  ds,
\end{align*}
and $\forall t\in\left]  \tau_{k},T\right]  $
\begin{align*}
p_{\underline{v}}\left(  t\right)   &  =C_{f}\exp\left(  -%
{\displaystyle\int\nolimits_{t}^{T}}
\alpha\left(  s\right)  /\left(  1-\sigma u\left(  s\right)  \right)
ds\right) \\
&  +%
{\displaystyle\int\nolimits_{t}^{T}}
\exp\left(  -%
{\displaystyle\int\nolimits_{t}^{s}}
\alpha\left(  \tau\right)  /\left(  1-\sigma u\left(  \tau\right)  \right)
d\tau\right)  ds.
\end{align*}
Using integration by parts we get
\begin{align*}
&
{\displaystyle\int\nolimits_{\tau_{k}}^{T}}
p_{\underline{v}}\left(  s\right)  dz_{v}\left(  s\right)  +%
{\displaystyle\sum\nolimits_{i=0}^{k-1}}
{\displaystyle\int\nolimits_{\tau_{i}}^{\tau_{i+1}}}
p_{\underline{v}}\left(  s\right)  dz_{v}\left(  s\right) \\
&  =p_{\underline{v}}\left(  T\right)  z_{v}\left(  T\right)  -p_{\underline{v}}\left(
\tau_{k}^{+}\right)  z_{v}\left(  \tau_{k}^{+}\right) \\
&  +%
{\displaystyle\sum\nolimits_{i=0}^{k-1}}
\left(  p_{\underline{v}}\left(  \tau_{i+1}\right)  z_{v}\left(  \tau_{i+1}\right)
-p_{\underline{v}}\left(  \tau_{i}^{+}\right)  z_{v}\left(  \tau_{i}^{+}\right)
\right) \\
&  -%
{\displaystyle\int\nolimits_{0}^{T}}
\left(  \alpha\left(  s\right)  p_{\underline{v}}\left(  s\right)  /\left(
1-\sigma u\left(  s\right)  \right)  -1\right)  z_{v}\left(  s\right)  ds.
\end{align*}
On the other hand, from (\ref{eq:deriv_z}) we get
\begin{align*}
&  {\displaystyle\int\nolimits_{\tau_{k}}^{T}}p_{\underline{v}}\left(  s\right)
\frac{dz_{v}}{ds}(s)\,ds+%
{\displaystyle\sum\nolimits_{i=0}^{k-1}}
{\displaystyle\int\nolimits_{\tau_{i}}^{\tau_{i+1}}}
p_{\underline{v}}\left(  s\right)  \frac{dz_{v}}{ds}\left(  s\right)  \,ds\\
&  =-%
{\displaystyle\int\nolimits_{0}^{T}}
\left(  \alpha\left(  s\right)  /\left(  1-\sigma u\left(  s\right)  \right)
\right)  p_{\underline{v}}\left(  s\right)  z_{v}\left(  s\right)  ds.
\end{align*}
Equating these two expressions yields (after rearrangement)
\begin{align*}
{\displaystyle\int\nolimits_{0}^{T}}  &  z_{v}\left(  s\right)  ds=-\left(
p_{\underline{v}}\left(  T\right)  z_{v}\left(  T\right)  -p_{\underline{v}}\left(
\tau_{k}^{+}\right)  z_{v}\left(  \tau_{k}^{+}\right)  \right) \\
&  ~~-%
{\displaystyle\sum\nolimits_{i=0}^{k-1}}
\left(  p_{\underline{v}}\left(  \tau_{i+1}\right)  z_{v}\left(  \tau_{i+1}\right)
-p_{\underline{v}}\left(  \tau_{i}^{+}\right)  z_{v}\left(  \tau_{i}^{+}\right)
\right)  .
\end{align*}
Plugging this into expression (\ref{eq:Jv}) for $J_{v}(\underline{v})$ and using
(\ref{EqAdjp2}) completes the proof of (\ref{eq:CondOptStrategy0}).

In the case where $\underline{v}$ is an optimal strategy $v^{\ast}$, then for an abitrary but fixed $v\in V$ and $\varepsilon>0$ sufficiently small, we
have $J\left(  v^{\ast}+\varepsilon v\right)  \geq J\left(  v^{\ast}\right)
,$ and consequently
\[
J_{v}(v^{\ast})=\lim_{\varepsilon\rightarrow0^{+}}\frac{1}{\varepsilon}\left(  J\left(
v^{\ast}+\varepsilon v\right)  -J\left(  v^{\ast}\right)  \right)  \geq0.
\]
\end{proof}
The result above is a version of the maximum principle which characterizes the
optimal strategy, but does not provide an efficient way to compute it. The following proposition provides a direct means for computing an optimal strategy in the case where $\sigma^{\ast}=0$.
\begin{proposition}
\label{eq:OptStrategyAverage2} There exists an optimal strategy $\underline{v}=v^{\ast}$ such that $v^{\ast}$ belongs to the set $\left\{  0,1\right\}  ^{k}$ and
\[
v_{i}^{\ast}=\left\{
\begin{array}
[c]{l}%
0\text{~~if }p_{v^{\ast}}\left(  \tau_{i}^{+}\right)  >c_{i}\text{ and }%
\Theta_{v^{\ast}}\left(  \tau_{i}\right)  \geq\sigma^{\ast},\\
1 \text{~~otherwise}.%
\end{array}
\right. %
\]

\end{proposition}

\begin{proof}
By Proposition~\ref{prop:exist}, we know that an optimal control strategy $v^{\ast}$ exists.
In the case where $p_{v^{\ast}}\left(  \tau_{i}%
^{+}\right)  > c_{i}$, then (\ref{eq:CondOptStrategy}) requires that  $v_{i} \in
\mathbb{R}_{+}$, which necessarily leads to $v_{i}^{\ast}=0$ (since $\Theta_{v^{\ast}
}\left(  \tau_{i}\right)  \geq\sigma^{\ast}>0$ by the definition of $\tau_i$). Similarly, if $p_{v^{\ast}%
}\left(  \tau_{i}^{+}\right)  <c_{i}$ then (\ref{eq:CondOptStrategy}) requires that $v_{i} \in
\mathbb{R}_{-}$, which corresponds necessarily to $v_{i}^{\ast}=1$.
If   $p_{v^{\ast}}\left(  \tau_{i} ^{+}\right)  = c_{i}$, then from (\ref{eq:CondOptStrategy}) we have that $J_{v_i}(v^{\ast})=0$. Supposing that $v_i^{\ast}<1$, we may increase $v_i^{\ast}$ without changing the value of $J(v^{\ast})$.  (Here we should note that increasing $v_i^{\ast}$ will not affect the condition
$\Theta_{v^{\ast}
}\left(  \tau_{j}\right)  \geq\sigma^{\ast}$ for $j \geq i$  because of monotonicity.)
   In particular, we may choose $v_i^{\ast}=1$ and still obtain a solution $v^{\ast}$ that minimizes the cost function.
\end{proof}
In the case where $\sigma^{\ast} = 0$, Proposition \ref{eq:OptStrategyAverage2} provides a constructive method for obtaining $v^{\ast}$, by using (\ref{EqAdjp1}) and (\ref{EqAdjp2}) to solve for $p_{v^{\ast}}(t)$ in a backwards time direction starting from $t=T$.  

\subsection{Space-dependent model with pulse strategy\label{COVGeneralModel}}

In this subsection we generalize the results of the the previous subsection by
establishing existence of and characterizing a strategy $v^{\ast}=\left(
v_{i}^{\ast}\right)  _{i\in%
\mathbb{N}
}$ which minimizes the following cost functional:%
\begin{align}
J\left(  v\right)   &  =%
{\displaystyle\int\nolimits_{0}^{T}}
{\displaystyle\int\nolimits_{\Omega}}
\theta\left(  s,x\right)  dxds+%
{\displaystyle\int\nolimits_{\Omega}}
C_{f}\left(  x\right)  \theta\left(  T,x\right)  dx\nonumber\\
&  +%
{\displaystyle\sum\nolimits_{i=0}^{k}}
{\displaystyle\int\nolimits_{\Omega}}
c_{i}\left(  x\right)  \left(  1-v_{i}\left(  x\right)  \right)  \theta\left(
\tau_{i},x\right)  dx, \label{eq:J_MainDef}%
\end{align}
where $\tau_{k}<T<\tau_{k+1}$; $\forall i\in%
\mathbb{N}
,$ $v_{i}\in L^{\infty}\left(  \Omega;\left[  0,1\right]  \right)  ,$
$c_{i}\in L^{\infty}\left(  \Omega;%
\mathbb{R}
_{+}\right)  $; and $\forall x\in\Omega$, $c(x)=\left(  c_{i}(x)\right)
_{i\in%
\mathbb{N}
}$ is a $%
\mathbb{R}
_{+}^{\ast}-$valued sequence of cost ratios related to the use of control.
$C_{f}$ is also a cost related to the final inhibition rate.

Before proving the existence of an optimal strategy, we recall the following
lemma stated in \cite{anita,brezis}.

\begin{lemma}
\label{lemma:Mazur} (Mazur)

Let $\left(  x_{n}\right)  _{n\in%
\mathbb{N}
}$ be a sequence taking its values in a real Banach space $X$ that is weakly
convergent to $x\in X$. Then there exists a $X-$valued sequence $\left(
y_{n}\right)  _{n\in%
\mathbb{N}
}$ which converges strongly to $x$ and such that $\forall n\in%
\mathbb{N}
,$ $y_{n}$ is an element of the convex hull of $\left(  x_{n}\right)  _{n\in%
\mathbb{N}
}$.
\end{lemma}

\begin{theorem}
\label{ExistenceOfAnOptimalStrategyv}There is an optimal strategy $v^{\ast
}=\left(  v_{i}^{\ast}\right)  _{i\in%
\mathbb{N}
}$ which minimizes $J(v)$.
\end{theorem}

\begin{proof}
The problem can be reduced to finding $v^{\ast}=\left(  v_{i}^{\ast}\right)
_{0\leq i\leq k}\in\left[  0,1\right]  ^{k+1}$, since the terms $v_{i}$ with
$i>k$ have no effect on $J$. Note that
\[
0\leq J(v)\leq\left\vert \Omega\right\vert \left(  1+T+%
{\displaystyle\sum\nolimits_{i=0}^{k}}
\left\Vert c_{i}\right\Vert _{L^{\infty}\left(  \Omega;%
\mathbb{R}
\right)  }\right)  .
\]
Let $J^{\ast}=\underset{v}{\inf}\,J\left(  v\right)  $. There is a $\left(
L^{\infty}\left(  \Omega;\left[  0,1\right]  \right)  \right)  ^{k+1}-$valued
sequence $\left\{  v^{n}\right\}  _{n\in%
\mathbb{N}
}$ such that the sequence $\left\{  J\left(  v^{n}\right)  \right\}
_{n\in%
\mathbb{N}
}$ converges to $J^{\ast}$. The sequence $\left\{  v^{n}\right\}  _{n\in%
\mathbb{N}
}$ is bounded and there is a subsequence $\left\{  v^{n_{m}}\right\}  $ which
converges weakly to a strategy $v^{\ast}\in\left(  L^{\infty}\left(
\Omega;\left[  0,1\right]  \right)  \right)  ^{k+1}$. Lemma \ref{lemma:Mazur}
implies there is a sequence $\left\{  \underline{v}^{n}\right\}  _{n\in%
\mathbb{N}
}$ in $conv\left(  \left\{  v^{n}\right\}  _{n\in%
\mathbb{N}
}\right)  \subset\left(  L^{\infty}\left(  \Omega;\left[  0,1\right]  \right)
\right)  ^{k+1}$ which converges strongly to $v^{\ast}$. Since $J(v)$ is
continuous, it follows that $J\left(  v^{\ast}\right)  =J^{\ast}$.

In the remainder of this subsection we characterize the optimal control
strategy in order to compute it. The solution of $\left(  \ref{ModelPulse1}%
\right)  -\left(  \ref{ModelPulse4}\right)  $ is continuous with respect to
control strategies $\left(  v_{i}\right)  $. Now we add the assumption that
the solution of $\left(  \ref{ModelPulse1}\right)  -\left(  \ref{ModelPulse4}%
\right)  $ is G\^{a}teaux differentiable with respect to $v=\left\{
v_{i}\right\}  _{0\leq i\leq k}\in\left(  L^{\infty}\left(  \Omega;\left[
0,1\right]  \right)  \right)  ^{k+1}$. Let $\theta_{\underline{v}}$ be the
solution of $\left(  \ref{ModelPulse1}\right)  -\left(  \ref{ModelPulse4}%
\right)  $ associated to a chosen control strategy $\underline{v}$ and let
$z_{v}$ be its directional derivative, $z_{v}=D_{v}\theta_{\underline{v}}$.
%
In analogy to the derivation of (\ref{eq:deriv_z}) and
(\ref{eq:zv_tau_averaged}) in the previous section, we may show that
$\forall x\in\Omega$ and $\forall i\in%
\mathbb{N}
^{\ast},$%
\begin{align}
z_{v}\left(  0^{+},x\right)   &  =v_{0}(x)\rho(x);\label{EqZv2PDE}\\
z_{v}\left(  \tau_{i}^{+},x\right)   &  =v_{i}\left(  x\right)  \theta
_{\underline{v}}\left(  \tau_{i},x\right)  +\underline{v}_{i}\left(
x\right)  z_{v}\left(  \tau_{i},x\right)  ; \label{EqZv3PDE}%
\end{align}
and $\forall t\in%
\mathbb{R}
_{+}^{\ast}\setminus\left(  \tau_{i}\right)  _{i\in%
\mathbb{N}
}$
\begin{align}
\partial_{t}z_{v}\left(  t,x\right)   &  =-\alpha\left(  t,x\right)
z_{v}\left(  t,x\right)  /\left(  1-\sigma u\left(  t,x\right)  \right)
\label{EqZv1PDE}\\
&  +\operatorname{div}\left(  A\left(  t,x\right)  \nabla z_{v}\left(
t,x\right)  \right)  ,\text{ }\nonumber
\end{align}

\end{proof}

As in the previous section, we may define $J_{v}(\underline{v})$ as the directional
derivative of the cost functional $J(\underline{v})$ as defined in (\ref{eq:J_MainDef}). A
straightforward computation yields
\begin{align}
J_{v}(\underline{v})  &  =%
{\displaystyle\int\nolimits_{\Omega}}
C_{f}\left(  x\right)  z_{v}\left(  T,x\right)  dx+%
{\displaystyle\int\nolimits_{0}^{T}}
{\displaystyle\int\nolimits_{\Omega}}
z_{v}\left(  s,x\right)  dxds\nonumber\\
&  +%
{\displaystyle\sum\nolimits_{i=0}^{k}}
{\displaystyle\int\nolimits_{\Omega}}
c_{i}\left(  x\right)  \left(  \left(  1-\underline{v}_{i}\left(  x\right)
\right)  z_{v}\left(  \tau_{i},x\right)  \right. \nonumber\\
&  \left.  -v_{i}\left(  x\right)  \theta_{\underline{v}}\left(  \tau_{i},x\right)
\right)  dx. \label{J_vMain}%
\end{align}

We may then state the following analogy of Proposition
\ref{eq:OptStrategyAverage}:

\begin{theorem}
\label{MainTheoremPulse} If $\underline{v}=v^{\ast}$ is an optimal strategy
then $\forall v\in V,$
\[%
{\displaystyle\sum\nolimits_{i=0}^{k}}
{\displaystyle\int\nolimits_{\Omega}}
\left(  p_{v^{\ast}}\left(  \tau_{i}^{+},x\right)  -c_{i}\left(  x\right)
\right)  v_{i}\left(  x\right)  \theta_{v^{\ast}}\left(  \tau_{i},x\right)
dx\geq0,
\]
where
\begin{align*}
V=  &  \left\{  v\in\left(  L^{\infty}\left(  \Omega;%
\mathbb{R}
\right)  \right)  ^{k+1};\right. \\
&  ~~\left.  \exists\varepsilon>0\,|\,v^{\ast}+\varepsilon v\in\left(
L^{\infty}\left(  \Omega;\left[  0,1\right]  \right)  \right)  ^{k+1}\right\}
,
\end{align*}
and $p_{v^{\ast}}$ is solution of the following adjoint problem $(\forall
x\in\Omega)$:
\begin{align}
\partial_{t}p_{v^{\ast}}  &  =\alpha\left(  t,x\right)  p_{v^{\ast}}/\left(
1-\sigma u\left(  t,x\right)  \right) \nonumber\\
&  \qquad\quad-\operatorname{div}\left(  A\left(  t,x\right)  \nabla
p_{v^{\ast}}\left(  t,x\right)  \right)  -1,\nonumber\\
&  \qquad t \in\left]  0,T\right]  \setminus\left\{  \tau_{i}\right\}  ,\text{
}i\in\left[  0,k\right]  \cap%
\mathbb{N}
^{\ast}\text{ };\label{EqAdjp1PDE}\\
p_{v^{\ast}}\left(  T,x\right)   &  =C_{f}\left(  x\right)
,\label{EqAdjp2PDE}\\
\text{ }p_{v^{\ast}}\left(  \tau_{i},x\right)   &  =v_{i}^{\ast}\left(
x\right)  p_{v^{\ast}}\left(  \tau_{i}^{+},x\right)  +c_{i}\left(  x\right)
\left(  1-v_{i}^{\ast}\left(  x\right)  \right)  ;
\end{align}
and
\begin{equation}
\left\langle A\left(  t,x\right)  \nabla p_{v^{\ast}}\left(  t,x\right)
,n\left(  x\right)  \right\rangle =0,\text{ on }%
\mathbb{R}
_{+}^{\ast}\times\partial\Omega. \label{EqAdjp3PDE}%
\end{equation}

\end{theorem}

\begin{proof}
The proof resembles that of Proposition~\ref{eq:OptStrategyAverage}. We
first argue as in Proposition \ref{ExistenceGlobalSolutionIPDE} that there is
a unique absolutely solution $p_{v^{\ast}}$ to the problem $\left(
\ref{EqAdjp1PDE}\right)  -\left(  \ref{EqAdjp3PDE}\right)  $. We may then
evaluate
\[%
{\displaystyle\int\nolimits_{0}^{T}}
p_{v^{\ast}}\left(  s,x\right)  \partial_{t}z_{v}\,ds+\sum_{i=0}%
^{k-1}{\displaystyle\int\nolimits_{\tau_{i}}^{\tau_{i+1}}}p_{v^{\ast}}\left(
s,x\right)  \partial_{t}z_{v}\,ds
\]
by parts, and alternatively by using expression (\ref{EqZv1PDE}) for
$\partial_{t}z_{v}$. After equating the two expressions, integrating over
$\Omega$, and rearranging we obtain
\begin{align*}
&
{\displaystyle\int\nolimits_{0}^{T}}
{\displaystyle\int\nolimits_{\Omega}}
z_{v}\left(  s,x\right)  dxds\\
&  =%
{\displaystyle\int\nolimits_{\Omega}}
p_{v^{\ast}}\left(  T,x\right)  z_{v}\left(  T,x\right)  dx\\
&  -%
{\displaystyle\sum\nolimits_{i=0}^{k-1}}
{\displaystyle\int\nolimits_{\Omega}}
p_{v^{\ast}}\left(  \tau_{i+1},x\right)  z_{v}\left(  \tau_{i+1},x\right)
dx\\
&  +%
{\displaystyle\sum\nolimits_{i=0}^{k}}
{\displaystyle\int\nolimits_{\Omega}}
p_{v^{\ast}}\left(  \tau_{i}^{+},x\right)  z_{v}\left(  \tau_{i}^{+},x\right)
dx,
\end{align*}
and substituting this expression into expression (\ref{J_vMain}) for $J_{v}(v^{\ast})$,
we find after simplification that
\[
J_{v}(v^{\ast})=%
{\displaystyle\sum\nolimits_{i=0}^{k}}
{\displaystyle\int\nolimits_{\Omega}}
\left(  p_{v^{\ast}}\left(  \tau_{i}^{+},x\right)  -c_{i}\left(  x\right)
\right)  v_{i}\left(  x\right)  \theta_{v^{\ast}}\left(  \tau_{i},x\right)
dx.
\]
For an arbitrary but fixed $v\in V$ and $\varepsilon>0$ sufficiently small we
have $J\left(  v^{\ast}+\varepsilon v\right)  \geq J\left(  v^{\ast}\right)  $
and consequently (as before) $J_{v}(v^{\ast})\geq0$.

\end{proof}

Consider in $\left(  L^{\infty}\left(  \Omega;%
\mathbb{R}
\right)  \right)  ^{k+1}\supset\left(  L^{\infty}\left(  \Omega;\left[
0,1\right]  \right)  \right)  ^{k+1}$ the inner product%
\[
\left\langle v_{1},v_{2}\right\rangle =%
{\displaystyle\sum\nolimits_{i=0}^{k}}
{\displaystyle\int\nolimits_{\Omega}}
v_{1}^{i}\left(  x\right)  v_{2}^{i}\left(  x\right)  ds
\]
with its associated norm
\[
\left\Vert v\right\Vert _{2}=%
{\displaystyle\sum\nolimits_{i=0}^{k}}
\left(
{\displaystyle\int\nolimits_{\Omega}}
\left(  v_{1}^{i}\left(  x\right)  \right)  ^{2}ds\right)  ^{1/2}.
\]
Note that $\forall v\in\left(  L^{\infty}\left(  \Omega;%
\mathbb{R}
\right)  \right)  ^{k+1},$
\[
\left\Vert v\right\Vert _{2}\leq\sqrt{\left\vert \Omega\right\vert }\times%
{\displaystyle\sum\nolimits_{i=0}^{k}}
\underset{x\in\Omega}{ess\sup}\left\vert v^{i}\left(  x\right)  \right\vert .
\]
This allows us to find $v^{\ast}$ using the topology given by $\left\Vert
.\right\Vert _{2}$. Since $J_{v}(v^{\ast})$ is a continuous linear operator
acting on $v$, there is a unique $\overline{v}\in\left(  L^{\infty}\left(
\Omega;%
\mathbb{R}
\right)  \right)  ^{k+1}$ such that%

\[
\left\langle \overline{v},v\right\rangle =J_{v}(v^{\ast}).
\]
From Theorem~\ref{MainTheoremPulse} we find $\forall j\in\left[  0,k\right]
\cap%
\mathbb{N}
,$ $\forall x\in\Omega,$
\[
\overline{v}_{j}\left(  x\right)  =\left(  p_{v^{\ast}}\left(  \tau_{j}%
^{+},x\right)  -c_{j}\left(  x\right)  \right)  \theta_{v^{\ast}}\left(
\tau_{j},x\right)  .
\]
This identification is useful in the implementation of the gradient method.

We close this section with a proposition which, like Proposition~\ref{eq:OptStrategyAverage2}, enables the direct computation of the optimal strategy when $\sigma^{\ast} = 0$.
\begin{proposition}\label{eq:OptStrategyPulseFull}
If $\underline{v}=v^{\ast}$ is an optimal
strategy then $\forall x\in\Omega,v^{\ast}\left(  x\right)  \in\left\{
0,1\right\}  ^{k}$. Moreover, there exists an optimal strategy $v_{i}^{\ast}$ such that
\[
v_{i}^{\ast}\left(  x\right)  =\left\{
\begin{array}
[c]{l}%
0,\text{ if }p_{v^{\ast}}\left(  \tau_{i}^{+},x\right)  >c_{i}\text{ }\\
\text{and }\left\Vert \theta\left(  \tau_{i},.\right)  \right\Vert
_{L^{2}\left(  \Omega\right)  }\geq\sigma^{\ast}\left\vert \Omega\right\vert,
\\
\\
1,\text{ otherwise}.%
\end{array}
\right.
\]

\end{proposition}

\begin{proof}
The proof is similar to that of Proposition
\ref{eq:OptStrategyAverage2}.
\end{proof}

It follows from Proposition~\ref{eq:OptStrategyPulseFull} that the optimal strategy may be found when $\sigma^{\ast}=0$ by backwards-solving for $p$.

\section{Optimal control based on both continuous and pulse
stategies\label{SectionCOUV}}

In this section we consider the possibility of a continuous control strategy
employed along with the pulse control. As in the previous section, we restrict
the time of study to the set $\left[  0,T\right]  $ where $T\in\left[
\tau_{k},\tau_{k+1}\right[  $ corresponds practically to an annual production period.

\subsection{Averaged model with mixed strategy\label{COUVAgragatedModel}}

In this subsection we demonstrate existence and provide a characterization of
a strategy $\left(  u^{\ast},v^{\ast}\right)  =\left(  u^{\ast},\left(
v_{i}^{\ast}\right)  _{i\in%
\mathbb{N}
}\right)  $ which minimizes the following cost functional:%
\begin{align*}
J\left(  u,v\right)   &  =%
{\displaystyle\int\nolimits_{0}^{T}}
\Theta\left(  s\right)  +C\left(  s\right)  u\left(  s\right)  ds\\
&  +%
{\displaystyle\sum\nolimits_{i=0}^{k}}
c_{i}\left(  1-v_{i}\right)  \Theta\left(  \tau_{i}\right)  +C_{f}%
\Theta\left(  T\right)  ,
\end{align*}
where $C\in L_{loc}^{\infty}\left(
\mathbb{R}
_{+};%
\mathbb{R}
_{+}\right)  $ is almost everywhere positive and $c=\left(  c_{i}\right)
_{i\in%
\mathbb{N}
}\subset%
\mathbb{R}
_{+}^{\ast}$. $C$ and $c$ are time dependent cost ratios related to the use of
the control strategy. It is evident that only the $k+1$ first terms of $v$ and
$c$ are significant. $C_{f}$ is a cost related to the final inhibition rate.
We also require that $u\in L^{\infty}\left(  \left[  0,T\right]  ;\left[
0,1\right]  \right)  $.

Existence of an optimal strategy is guaranteed by the following proposition,
which uses virtually the same argument as
Theorem~\ref{ExistenceOfAnOptimalStrategyv}.

\begin{proposition}
\label{ExistenceOfAnOptimalStrategyuv}There is an optimal strategy $\left(
u^{\ast},v^{\ast}\right)  =\left(  u^{\ast},\left(  v_{i}^{\ast}\right)
_{i\in%
\mathbb{N}
}\right)  $ which minimizes $J$.
\end{proposition}

\begin{proof}
We may take $v^{\ast}=\left(  v_{i}^{\ast}\right)  _{0\leq i\leq k}\in\left[
0,1\right]  ^{k+1}$, since $v_{i}$ with $i>k$ do not affect $J$. Note that
\[
0\leq J\leq1+\left(  1+\underset{s\in\left[  0,T\right]  }{\sup}C\left(
s\right)  \right)  T+%
{\displaystyle\sum\nolimits_{i=0}^{k}}
c_{i}.
\]
Let
\[
J^{\ast}=\underset{\left(  u,v\right)  \in L^{\infty}\left(  \left[
0,T\right]  ;\left[  0,1\right]  \right)  \times\left[  0,1\right]  ^{k+1}%
}{\inf}J\left(  u,v\right)  .
\]
Then there is a $L^{\infty}\left(  \left[  0,T\right]  ;\left[  0,1\right]
\right)  \times\left[  0,1\right]  ^{k+1}-$valued sequence $\left\{
u^{n},v^{n}\right\}  _{n\in%
\mathbb{N}
}$ such that the sequence $\left\{  J\left(  u^{n},v^{n}\right)  \right\}
_{n\in%
\mathbb{N}
}$ converges to $J^{\ast}$. The sequence $\left\{  u^{n},v^{n}\right\}  _{n\in%
\mathbb{N}
}$ is bounded, so there is a subsequence $\left\{  u^{n_{m}},v^{n_{m}%
}\right\}  $ which converges weakly to a strategy $\left\{  u^{\ast},v^{\ast
}\right\}  \in L^{\infty}\left(  \left[  0,T\right]  ;\left[  0,1\right]
\right)  \times\left[  0,1\right]  ^{k+1}$. Using the lemma of Mazur, there is
a sequence $\left\{  \underline{u}^{n},\underline{v}^{n}\right\}  _{n\in%
\mathbb{N}
}$ in $conv\left(  \left\{  \left(  u^{n},v^{n}\right)  \right\}  _{n\in%
\mathbb{N}
}\right)  \subset L^{\infty}\left(  \left[  0,T\right]  ;\left[  0,1\right]
\right)  \times\left[  0,1\right]  ^{k+1}$ which converges strongly to
$\left(  u^{\ast},v^{\ast}\right)  $. Since $J$ is continuous, $J\left(
u^{\ast},v^{\ast}\right)  =J^{\ast}$.
\end{proof}

Now let $\Theta_{\underline{u},\underline{v}}$ be the solution of $\left(
\ref{ModelPulse1IDE}\right)  -\left(  \ref{ModelPulse3IDE}\right)  $
associated to a chosen control strategy $\left(  \underline{u},\underline
{v}\right)  $, and let $z_{u}=D_{u}\Theta_{\underline{u},\underline{v}}$
and $J_{u}(\underline{u},\underline{v})=D_{u}J\left(  \underline{u},\underline
{v}\right)$. Then using (\ref{eq:Phi}) we may
compute for $t\in]\tau_{i},\tau_{i+1}]$,
\begin{align*}
&  z_{u}\left(  t\right) \\
&  =\underline{v}_{i}\left(  z_{u}\left(  \tau_{i}\right)  -\sigma
\Theta_{\underline{u},\underline{v}}\left(  \tau_{i}\right)
{\displaystyle\int\nolimits_{\tau_{i}}^{t}}
\alpha\left(  s\right)  \underline{u}\left(  s\right)  /\left(  1-\sigma\underline
{u}\left(  s\right)  \right)  ^{2}ds\right) \\
&  \qquad\times\exp\left(  -%
{\displaystyle\int\nolimits_{\tau_{i}}^{t}}
\alpha\left(  s\right)  /\left(  1-\sigma\underline{u}\left(  s\right)
\right)  ds\right) \\
&  ~-%
{\displaystyle\int\nolimits_{\tau_{i}}^{t}}
\sigma\alpha\left(  s\right)  \exp\left(  -%
{\displaystyle\int\nolimits_{s}^{t}}
\alpha\left(  \tau\right)  /\left(  1-\sigma\underline{u}\left(  \tau\right)
\right)  d\tau\right) \\
&  \qquad\times%
{\displaystyle\int\nolimits_{s}^{t}}
\alpha\left(  \tau\right)  \underline{u}\left(  \tau\right)  /\left(  1-\sigma
\underline{u}\left(  \tau\right)  \right)  ^{2}d\tau ds.
\end{align*}
In analogy with (\ref{eq:zv_computation}), we may derive for $t\in]\tau
_{i},\tau_{i+1}]$
\begin{align*}
&  z_{u}\left(  t\right) \\
&  =-%
{\displaystyle\sum\nolimits_{j=0}^{i-1}}
\left(
{\displaystyle\prod\nolimits_{l=j}^{i}}
\underline{v}_{l}\right)
{\displaystyle\int\nolimits_{\tau_{j}}^{\tau_{j+1}}}
\alpha\left(  s\right)  \underline{u}\left(  s\right)  /\left(  1-\sigma\underline
{u}\left(  s\right)  \right)  ^{2}ds\\
&  \qquad\times\sigma\Theta_{\underline{u},\underline{v}}\left(  \tau
_{j}\right)  \exp\left(  -%
{\displaystyle\int\nolimits_{\tau_{j}}^{t}}
\underline{u}\left(  s\right)  \alpha\left(  s\right)  /\left(  1-\sigma\underline
{u}\left(  s\right)  \right)  ds\right) \\
&  ~-%
{\displaystyle\sum\nolimits_{j=0}^{i-1}}
\left(
{\displaystyle\prod\nolimits_{l=j+1}^{i}}
\underline{v}_{l}\right) \\
&  \qquad\times%
{\displaystyle\int\nolimits_{\tau_{j}}^{\tau_{j+1}}}
{\displaystyle\int\nolimits_{s}^{\tau_{j+1}}}
\alpha\left(  \tau\right)  \underline{u}\left(  \tau\right)  /\left(  1-\sigma
\underline{u}\left(  \tau\right)  \right)  ^{2}d\tau\\
&  \qquad\times\sigma\alpha\left(  s\right)  \exp\left(  -%
{\displaystyle\int\nolimits_{s}^{t}}
\alpha\left(  \tau\right)  /\left(  1-\sigma\underline{u}\left(  \tau\right)
\right)  d\tau\right)  ds
\end{align*}%
\begin{align*}
&  ~-%
{\displaystyle\int\nolimits_{\tau_{i}}^{t}}
{\displaystyle\int\nolimits_{s}^{t}}
\alpha\left(  \tau\right)  \underline{u}\left(  \tau\right)  /\left(  1-\sigma
\underline{u}\left(  \tau\right)  \right)  ^{2}d\tau\\
&  \qquad\times\sigma\alpha\left(  s\right)  \exp\left(  -%
{\displaystyle\int\nolimits_{s}^{t}}
\alpha\left(  \tau\right)  /\left(  1-\sigma\underline{u}\left(  \tau\right)
\right)  d\tau\right)  ds\\
&  ~-%
{\displaystyle\int\nolimits_{\tau_{i}}^{t}}
\alpha\left(  s\right)  \underline{u}\left(  s\right)  /\left(  1-\sigma\underline
{u}\left(  s\right)  \right)  ^{2}ds\\
&  \qquad\times\underline{v}_{i}\sigma\Theta_{\underline{u},\underline{v}%
}\left(  \tau_{i}\right)  \exp\left(  -%
{\displaystyle\int\nolimits_{\tau_{i}}^{t}}
\alpha\left(  s\right)  /\left(  1-\sigma\underline{u}\left(  s\right)
\right)  ds\right)  .
\end{align*}
Thus $z_{u}$ satisfies:
\begin{equation}
z_{u}\left(  0\right)  =0;\qquad z_{u}\left(  \tau_{i}^{+}\right)
=\underline{v}_{i}z_{u}\left(  \tau_{i}\right)  , \label{EqZu3}%
\end{equation}
and $\forall t\in\mathbb{R}_{+}^{\ast}\setminus(\tau_{i})_{i\in\mathbb{N}}$,
\begin{align}
&  dz_{u}/dt=-\alpha\left(  t\right)  z_{u}\left(  t\right)  /\left(
1-\sigma\underline{u}\left(  t\right)  \right) \nonumber\\
&  \qquad-\sigma\alpha\left(  t\right)  \underline{u}\left(  t\right)  \Theta
_{\underline{u},\underline{v}}\left(  t\right)  /\left(  1-\sigma\underline
{u}\left(  t\right)  \right)  . \label{EqZu1}%
\end{align}
We may also compute
\begin{align}
J_{u}(\underline{u}.\underline{v})=%
& {\displaystyle\int\nolimits_{0}^{T}}
z_{u}\left(  s\right)    +C\left(  s\right)  \underline{u}\left(  s\right)
ds\nonumber\\
& \quad +%
{\displaystyle\sum\nolimits_{i=0}^{k}}
c_{i}\left(  1-\underline{v}_{i}\right)  z_{u}\left(  \tau_{i}\right)
+C_{f}z_{u}\left(  T\right)  . \label{EqJu1}%
\end{align}

\begin{proposition}
\label{prop:JuJvOpt1} If $\left(  \underline{u},\underline{v}\right)  =\left(
u^{\ast},v^{\ast}\right)  $ is an optimal strategy then $\forall\left(
u,v\right)  \in V,$%
\begin{align*}
0  &  \leq J_{u}(u^{\ast},v^{\ast})+J_{v}(u^{\ast},v^{\ast})\\
&  =%
{\displaystyle\int\nolimits_{0}^{T}}
C\left(  s\right)  u\left(  s\right)  ds+%
{\displaystyle\sum\nolimits_{i=0}^{k}}
\left(  p_{u^{\ast}}\left(  \tau_{i}^{+}\right)  -c_{i}\right)  v_{i}%
\Theta_{v^{\ast}}\left(  \tau_{i}\right) \\
&  -%
{\displaystyle\int\nolimits_{0}^{T}}
\sigma u\left(  s\right)  \alpha\left(  s\right)  p_{u^{\ast}}\left(
s\right)  \Theta_{u^{\ast},v^{\ast}}\left(  s\right)  /\left(  1-\sigma
u^{\ast}\left(  s\right)  \right)  ^{2}ds,
\end{align*}
where%
\begin{align*}
V  &  =\left\{  \left(  u,v\right)  \in L^{\infty}\left(  \left[  0,T\right]
;%
\mathbb{R}
\right)  \times%
\mathbb{R}
^{k+1};\exists\varepsilon>0;\right. \\
&  \left.  \left(  u^{\ast}+\varepsilon u,v^{\ast}+\varepsilon v\right)  \in
L^{\infty}\left(  \left[  0,T\right]  ;\left[  0,1\right]  \right)
\times\left[  0,1\right]  ^{k+1}\right\}  ,
\end{align*}
and $p_{u^{\ast},v^{\ast}}$ is the solution to the problem $\left(
\ref{EqAdjp1}\right)  -\left(  \ref{EqAdjp2}\right)$  with $\underline{v}=v^{\ast}$ and $u = u^{\ast}$.
\end{proposition}

\begin{proof}
The proof parallels those of Proposition \ref{eq:OptStrategyAverage} and
Theorem \ref{MainTheoremPulse}. We may evaluate
\[
\sum_{i=0}^{k} \int_{\tau_{i}}^{\tau_{i+1}} p_{u^{\ast},v^{\ast}} dz_{u}(s) +
\int_{\tau_{k}}^{T}p_{u^{\ast},v^{\ast}} dz_{u}(s)
\]
in two ways: by parts, and using the derivative expression (\ref{EqZu1}).
Equating the two results gives an expression for $\int_{0}^{T} z_{u}(s) ds$,
which we may then plug into (\ref{EqJu1}) and simplify to obtain the above
expression for $J_{u}(u^{\ast},v^{\ast}) + J_{v}(u^{\ast},v^{\ast})$. The conclusion $J_{u}(u^{\ast},v^{\ast}) + J_{v}(u^{\ast},v^{\ast}) \ge0$ follows as
before from the fact that $J\left(  u^{\ast}+\varepsilon u,v^{\ast
}+\varepsilon v\right)  \geq J\left(  u^{\ast},v^{\ast}\right)  $ for abitrary
but fixed $\left(  u,v\right)  \in V$ and $\varepsilon>0$ sufficiently small.
\end{proof}

For $u_{1}, u_{2} \in L^{\infty}\left(  \left[  0,T\right]  ;%
\mathbb{R}
\right)  \supset L^{\infty}\left(  \left[  0,T\right]  ;\left[  0,1\right]
\right)  , $ define the inner product%
\[
\left\langle u_{1},u_{2}\right\rangle =%
{\displaystyle\int\nolimits_{0}^{T}}
u_{1}\left(  s\right)  u_{2}\left(  s\right)  ds
\]
with its associated norm $\left\Vert u\right\Vert _{3} =\left(
{\displaystyle\int\nolimits_{0}^{T}}
u^{2}\left(  s\right)  ds\right)  ^{1/2}$. Note that $\forall u\in L^{\infty
}\left(  \left[  0,T\right]  ;%
\mathbb{R}
\right)  ,$ $\left\Vert u\right\Vert \leq\sqrt{T}\times\underset{t\in\left[
0,T\right]  }{ess\sup}\left\vert u\left(  t\right)  \right\vert _{3} $; thus
we may find $u^{\ast}$ using the topology given by $\left\Vert .\right\Vert
_{3}$. Since $J_{u}(u^{\ast},v^{\ast})$ is a linear continuous operator acting on $u$, it follows
there is a unique $\overline{u}\in L^{\infty}\left(  \left[  0,T\right]  ;%
\mathbb{R}
\right)  $ such that%

\begin{equation} \label{eq:j_u}
\left\langle \overline{u},u\right\rangle =J_{u}(u^{\ast},v^{\ast}).
\end{equation}
Indeed, from Proposition~\ref{prop:JuJvOpt1} we find
\begin{equation}\label{eq:u_overline}
\overline{u}=C-\sigma\alpha p_{u^{\ast},v^{\ast}}\Theta_{u^{\ast}%
,v^{\ast}}/\left(  1-\sigma u^{\ast}\right)  ^{2}%
\end{equation}

\begin{proposition}
\label{prop:JuJvOpt12}If $\left(  \underline{u},\underline{v}\right)  =\left(
u^{\ast},v^{\ast}\right)  $ is an optimal strategy then $v^{\ast}\in\left\{
0,1\right\}  ^{k}$. Moreover,
\begin{equation}
v_{i}^{\ast}=\left\{
\begin{array}
[c]{l}%
0,\text{ if }p_{u^{\ast},v^{\ast}}\left(  \tau_{i}^{+}\right)  >c_{i}\text{ and }%
\Theta_{u^{\ast},v^{\ast}}\left(  \tau_{i}\right)  \geq\sigma^{\ast}\\
1,\text{ otherwise},%
\end{array}
\right.
\end{equation}
and%
\begin{equation}\label{eq:Optucondition}
u^{\ast}=\left\{
\begin{array}
[c]{l}%
0,\text{ if }C>\sigma\alpha p_{u^{\ast},v^{\ast}}\Theta_{u^{\ast},v^{\ast}}/\left(  1-\sigma\right)\\
1,\text{otherwise. }
\end{array}
\right.
\end{equation}

\end{proposition}

\begin{proof}
The proof is similar to that of Proposition
\ref{eq:OptStrategyAverage2}.  The extremality of the pulse strategy follows as before. 
To show extremality of the optimum continuous strategy, we suppose that $(u^{\ast},v^{\ast})$ is an optimum strategy, and for real parameters $(\zeta, s, \delta)$ define an alternate continuous control strategy
\begin{equation*}
u_{(\zeta,s,\delta)}(t) \equiv u^{\ast}(t) + \zeta\chi_{[s,s+\delta]}(t), 
\end{equation*}
where $\chi_{[a,b]}$ denotes the characteristic function for the interval $[a,b]$. We may then compute
\begin{equation}\label{eq:Ju}
\begin{aligned}
J(u_{(\zeta,s,\delta)},v^{\ast}) - J(u^{\ast},v^{\ast}) &= \int_0^{\zeta} J_{\chi_{[s,s+\delta]}} (u^{\ast} + \xi\chi_{[s,s+\delta]}, v^{\ast}) d\xi \\
&= \int_0^{\zeta} \langle \overline{u}_{\xi,s,\delta}\chi_{[s,s+\delta]} \rangle d\xi,
\end{aligned}
\end{equation}
where
\begin{equation*}
 \overline{u}_{(\zeta,s,\delta)} \equiv C - \frac{\sigma \alpha p_{u^{\ast},v^{\ast}}\Theta_{u^{\ast},v^{\ast}}}{\left(1 - \sigma(u^{\ast} + \zeta\chi_{[s,s+\delta]})\right)^2}
\end{equation*}
Substituting into (\ref{eq:Ju}), we find (for  $\delta \rightarrow 0$)
\begin{align*}
&J(u_{(\zeta,s,\delta)},v^{\ast}) - J(u^{\ast},v^{\ast}) \\
&\qquad=
 \delta \int_0^{\zeta}  
C - \frac{\sigma \alpha p_{u^{\ast},v^{\ast}}\Theta_{u^{\ast},v^{\ast}}}{\left(1 - \sigma(u^{\ast}(s) + \zeta)\right)^2} d\xi
+ \mathcal{O}(\delta^2) \\
&\qquad= \delta \left(q(u^{\ast}(s)+\zeta) - q(u^{\ast}(s)) \right) + \mathcal{O}(\delta^2),
\end{align*}
where
\begin{equation*}
q(\eta) \equiv c \eta - \frac{\alpha p_{u^{\ast},v^{\ast}}\Theta_{u^{\ast},v^{\ast}}}{1- \sigma\eta}.
\end{equation*}
Since $J$ assumes its minimum at $(u^{\ast},v^{\ast})$, it follows that  $q(\eta)$ is minimized when $\eta = u^{\ast}(s)$  (note that $0 \le \eta \le 1$ because of restrictions on the control $u(s)$). It is easily shown that $q'' < 0$ on that interval, which implies that the minimum occurs at one of the endpoints, that is $\eta = 0$ or $\eta = 1$.  To find which, we compute
\begin{equation}
q(1) - q(0) = C - \frac{\sigma \alpha p_{u^{\ast},v^{\ast}}}{1 - \sigma},
\end{equation}
which leads directly to the condition (\ref{eq:Optucondition}).
%
\end{proof}

\subsection{Space-dependent model with mixed strategy\label{COUVGeneralModel}}

In this subsection we survey the existence and we characterize of a strategy
$\left(  u^{\ast},v^{\ast}\right)  =\left(  u^{\ast},\left(  v_{i}^{\ast
}\right)  _{i\in%
\mathbb{N}
}\right)  $ for the main model which minimizes the following cost functional:%
\begin{align*}
J\left(  u,v\right)   &  =%
{\displaystyle\int\nolimits_{\Omega}}
{\displaystyle\int\nolimits_{0}^{T}}
\theta\left(  s,x\right)  +C\left(  s,x\right)  u\left(  s,x\right)  dsdx\\
&  +%
{\displaystyle\sum\nolimits_{i=0}^{k}}
{\displaystyle\int\nolimits_{\Omega}}
c_{i}\left(  x\right)  \left(  1-v_{i}\left(  x\right)  \right)  \theta\left(
\tau_{i},x\right)  dx\\
&  +%
{\displaystyle\int\nolimits_{\Omega}}
C_{f}\left(  x\right)  \theta\left(  T,x\right)  dx,
\end{align*}
where $C\in L_{loc}^{\infty}\left(
\mathbb{R}
_{+}\times\Omega;%
\mathbb{R}
_{+}\right)  $ is almost everywhere positive, $\forall i\in%
\mathbb{N}
,$ $v_{i}\in L^{\infty}\left(  \Omega;\left[  0,1\right]  \right)  ,$
$c_{i}\in L^{\infty}\left(  \Omega;%
\mathbb{R}
_{+}\right)  ,$ $\forall x\in\Omega$, $c\left(  x\right)  =\left(
c_{i}\left(  x\right)  \right)  _{i\in%
\mathbb{N}
}\subset%
\mathbb{R}
_{+}^{\ast}$. $C$ and $c$ are time dependent cost ratios related to the use of
the control strategy. $C_{f}$ is a cost related to the final inhibition rate.
As before, we need only consider the first $k+1$ terms of $v$ and $c$. We may
also consider the restriction of $u$ to $\left[  0,T\right]  $, so that $u\in
L^{\infty}\left(  \left[  0,T\right]  \times\Omega;\left[  0,1\right]
\right)  $.

The following theorem generalizes Theorem \ref{ExistenceOfAnOptimalStrategyv} %
and Proposition \ref{ExistenceOfAnOptimalStrategyuv}.

\begin{theorem}
\label{ExistenceOfAnOptimalStrategyuvPDE}There is an optimal strategy $\left(
u^{\ast},v^{\ast}\right)  =\left(  u^{\ast},\left(  v_{i}^{\ast}\right)
_{i\in%
\mathbb{N}
}\right)  $ which minimizes $J$.
\end{theorem}

\begin{proof}
The proof uses the same argument as in Theorem
\ref{ExistenceOfAnOptimalStrategyv} and Proposition
\ref{ExistenceOfAnOptimalStrategyuv}, based on the lemma of Mazur.
\end{proof}

Let $\theta_{\underline{u},\underline{v}}$ be the solution of $\left(
\ref{ModelPulse1}\right)  -\left(  \ref{ModelPulse4}\right)  $ associated to a
chosen control strategy $\left(  \underline{u},\underline{v}\right)  $, and
let $z_{u}=D_{u}\theta_{\underline{u},\underline{v}}$ and $J_{u}=D_{u}J\left(
\underline{u},\underline{v}\right)  $. Using methods we have demonstrated
above, it may be shown that for $(t,x)\in\left(  \mathbb{R}_{+}^{\ast
}\setminus\left(  \tau_{i}\right)  _{i\in\mathbb{N}}\right)  \times\Omega$,
\begin{align}
\partial_{t}z_{u}=  &  -\alpha\left(  t,x\right)  z_{u}\left(  t,x\right)
/\left(  1-\sigma\underline{u}\left(  t,x\right)  \right) \nonumber\\
&  -\sigma\alpha\left(  t,x\right)  u\left(  t,x\right)  \theta_{\underline
{u},\underline{v}}\left(  t,x\right)  /\left(  1-\sigma\underline{u}\left(
t,x\right)  \right) \nonumber\\
&  +\operatorname{div}\left(  A\left(  t,x\right)  \nabla z_{u}\left(
t,x\right)  \right)  ,\text{ } \label{EqZu1PDE}%
\end{align}
where $z_{u}$ additionally satisfies
\begin{equation}
z_{u}\left(  0,x\right)  =0,\text{ }x\in\Omega\label{EqZu2PDE}%
\end{equation}
and
\begin{equation}
z_{u}\left(  \tau_{i}^{+},x\right)  =\underline{v}_{i}\left(  x\right)
z_{u}\left(  \tau_{i},x\right)  ,\text{ }x\in\Omega. \label{EqZu3PDE}%
\end{equation}
Furthermore, we may show
\begin{align*}
J_{u}  &  =%
{\displaystyle\int\nolimits_{0}^{T}}
{\displaystyle\int\nolimits_{\Omega}}
z_{u}\left(  s,x\right)  +C\left(  s,x\right)  u\left(  s,x\right)  dxds\\
&  +%
{\displaystyle\sum\nolimits_{i=0}^{k}}
{\displaystyle\int\nolimits_{\Omega}}
c_{i}\left(  x\right)  \left(  1-\underline{v}_{i}\left(  x\right)  \right)
z_{u}\left(  \tau_{i},x\right)  dx\\
&  +%
{\displaystyle\int\nolimits_{\Omega}}
z_{u}\left(  T,x\right)  dx.
\end{align*}
We obtain finally the following generalization of

\begin{theorem}
\label{MainTheorem} If $\left(  \underline{u},\underline{v}\right)  =\left(
u^{\ast},v^{\ast}\right)  $ is an optimal strategy then $\forall\left(
u,v\right)  \in V,$%
\begin{align*}
0  &  \leq J_{u}+J_{v}\\
&  =%
{\displaystyle\int\nolimits_{0}^{T}}
{\displaystyle\int\nolimits_{\Omega}}
C\left(  s,x\right)  u\left(  s,x\right)  dxds\\
&  +%
{\displaystyle\sum\nolimits_{i=0}^{k}}
{\displaystyle\int\nolimits_{\Omega}}
\left(  p_{v^{\ast}}\left(  \tau_{i}^{+},x\right)  -c_{i}\left(  x\right)
\right)  v_{i}\left(  x\right) \\
&  \qquad\qquad\qquad\times\theta_{u^{\ast},v^{\ast}}\left(  \tau
_{i},x\right)  dx\\
&  -%
{\displaystyle\int\nolimits_{0}^{T}}
{\displaystyle\int\nolimits_{\Omega}}
\sigma u\left(  s,x\right)  \alpha\left(  s,x\right)  p_{u^{\ast}}\left(
s,x\right)  \theta_{u^{\ast},v^{\ast}}\left(  s,x\right) \\
&  \qquad\qquad/\left(  1-\sigma u^{\ast}\left(  s,x\right)  \right)
^{2}dxds,
\end{align*}
where%
\begin{align*}
\mathcal{L} &\equiv
L^{\infty}\left(  \left[  0,T\right]
\times\Omega;%
\mathbb{R}
\right)  \times\left(  L^{\infty}\left(  \Omega;%
\mathbb{R}
\right)  \right)  ^{k+1};\\
V&=\left\{  \left(  u,v\right)  \in \mathcal{L}\,|\,\exists\varepsilon>0,
  \left(  u^{\ast}+\varepsilon u,v^{\ast}+\varepsilon v\right)  \in
\mathcal{L}\right\},
\end{align*}
and $p_{u^{\ast}}=p_{v^{\ast}}$is the solution to the problem $\left(
\ref{EqAdjp1PDE}\right)  -\left(  \ref{EqAdjp3PDE}\right)  .$
\end{theorem}

\begin{proof}
The proof follows the same lines as Proposition \ref{eq:OptStrategyAverage},
Theorem \ref{MainTheoremPulse}, and Theorem~\ref{prop:JuJvOpt1}, albeit the
calculations are more complicated.
\end{proof}

For $u_{1},u_{2}\in L^{\infty}\left(  \left[  0,T\right]  \times\Omega;%
\mathbb{R}
\right)  \supset L^{\infty}\left(  \left[  0,T\right]  \times\Omega;\left[
0,1\right]  \right)  $, define the inner product%
\[
\left\langle u_{1},u_{2}\right\rangle =%
{\displaystyle\int\nolimits_{0}^{T}}
{\displaystyle\int\nolimits_{\Omega}}
u_{1}\left(  s,x\right)  u_{2}\left(  s,x\right)  dxds,
\]
with associated norm $\left\Vert u\right\Vert _{4}=\left(
{\displaystyle\int\nolimits_{0}^{T}}
{\displaystyle\int\nolimits_{\Omega}}
u^{2}\left(  s,x\right)  dxds\right)  ^{1/2}$. Note that $\forall u\in
L^{\infty}\left(  \left[  0,T\right]  \times\Omega;%
\mathbb{R}
\right)  ,$ $\left\Vert u\right\Vert _{4}$ is less than $\sqrt{\left\vert
\Omega\right\vert T}\underset{t\in\left[  0,T\right]  }{ess\sup}\left\vert
u\left(  t,x\right)  \right\vert $; this allows us to find $u^{\ast}$ using
the topology given by $\left\Vert .\right\Vert _{4}$.

Since $J_{u}$ is a linear continuous operator acting on $u$, it follows there
is a unique $\overline{u}\in L^{\infty}\left(  \left[  0,T\right]
\times\Omega;%
\mathbb{R}
\right)  $ such that$\left\langle \overline{u},u\right\rangle =J_{u}$. Indeed,
from Theorem~\ref{MainTheorem} we have
\[
\overline{u}=C-\sigma\alpha p_{\underline{u}}\theta_{\underline{u}%
,\underline{v}}/\left(  1-\sigma\underline{u}\right)  ^{2}.
\]

\begin{proposition}
\label{prop:JuJvOpt13}If $\underline{v}=v^{\ast}$ is an optimal strategy then
$\forall x\in\Omega,v^{\ast}\left(  x\right)  \in\left\{  0,1\right\}  ^{k}$.
Moreover,
\[
v_{i}^{\ast}\left(  x\right)  =\left\{
\begin{array}
[c]{l}%
0, \text{ if }p_{v^{\ast}}\left(  \tau_{i}^{+},x\right)  >c_{i}\text{ and }\\
\qquad\left\Vert \theta_{u^{\ast},v^{\ast}}\left(  \tau_{i},.\right)  \right\Vert
_{L^{2}\left(  \Omega\right)  }\geq\sigma^{\ast}\left\vert \Omega\right\vert;
\\
1, \text{ otherwise},%
\end{array}
\right.
\]
and
\begin{equation}\label{eq:Optuconditionfull}
u^{\ast}(t,x)=\left\{
\begin{array}
[c]{l}%
0, \text{ if }C>\sigma\alpha p_{u^{\ast}}\theta_{u^{\ast},v^{\ast}}/\left(
1-\sigma\right); \\
1,\text{ if }C<\sigma\alpha p_{u^{\ast}}\theta_{u^{\ast},v^{\ast}}/\left(
1-\sigma\right),\\
\end{array}
\right.
\end{equation}
where $C$ is a function of $t$ and $\alpha, p_{u^{\ast}}$, and $\theta_{u^{\ast},v^{\ast}}$ are functions of $(t,x)$.
\end{proposition}

\begin{proof}
The proof is similar to that of Proposition
\ref{eq:OptStrategyAverage2}.
\end{proof}

\section{Model Simulation\label{SectionModSim}}

In this section we present simulations which verify that algorithms based on Propositions \ref{eq:OptStrategyAverage2} and \ref{eq:OptStrategyPulseFull} are effective in determing optimal pulse-only control strategies for the spatially-averaged and space-dependent models, respectively.

\subsection{Model discretization\label{SubsectionSimulPulseStrategy}}
First we briefly describe the implementation of the model.  The averaged model is a special case of the space-dependent model (on a $1 \times 1 \times 1$ grid), so it suffices to describe the space-dependent model. For the purposes of this description, we define 
\begin{align*}
\theta^{(0)} &\equiv \theta(t,x); \quad \theta^{(1)} \equiv \theta(t+h,x); \quad \theta^{(.5)} \equiv \theta(t+0.5h,x);\\
\alpha^{(.5)} &\equiv \alpha(t,x); \quad u^{(.5)} \equiv u(t,x); \quad A^{(.5)} \equiv A(t,x),
\end{align*}
where $h$ is the discrete time step used in the simulation. Using centered time-difference, we then have:
\begin{equation}\label{eq:disc1}
\frac{ \theta^{(1)} -  \theta^{(0)}}{h} \approx \alpha^{(.5)}(1 - \frac{ \theta^{(.5)}}{1 - \sigma u^{(.5)}} + \text{div} \left(A^{(0.5)} \nabla \theta^{(.5)}\right).
\end{equation}
We may also approximate:
\[ \theta^{(.5)} \approx \frac{\theta^{(0)} + \theta^{(1)}}{2},\]
which enables us to rewrite (\ref{eq:disc1} as:
\begin{equation}
 \theta^{(1)} \approx h\left( I - \frac{h}{2}M\right)^{-1} \left[ \alpha^{(.5)} +\left(I + \frac{h}{2}M \right) \theta^{(0)} \right],
\end{equation}
where $M$ represents the operator:
\begin{equation}
M \equiv  -\frac{\alpha^{(.5)}}{1 - \sigma u^{(.5)}}I + \text{div} \left(A^{(0.5)} \nabla\right).
\end{equation}
It remains to find a spatial discretization of $M$ that respects the boundary conditions. For simplicity we restricted ourselves to the case where $A$ is diagonal:  $A \equiv \text{diag}(A_1,A_2,A_3)$. 
 \begin{align*}
\text{div} \left(A^{(0.5)} \nabla \phi \right) 
&= \frac{\partial}{\partial x_1} \left(A_1^{(0.5)}\frac{\partial}{\partial x_1} \phi\right)
+ \frac{\partial}{\partial x_2} \left(A_2^{(0.5)}\frac{\partial}{\partial x_2} \phi\right)\\
&~~~~+ \frac{\partial}{\partial x_3} \left(A_3^{(0.5)}\frac{\partial}{\partial x_3} \phi\right)\\
\end{align*}
Letting the subscripts $i,j,k$ denote the grid point indices in the $x_1,x_2,$ and $x_3$ directions respectively, we 
may discretize as follows:
\begin{align*}
\left( \text{div} \left(A^{(0.5)} \nabla \phi \right) \right)_{i,j,k}
&= \frac{1}{ds^2} \Big( A_{1,(i+.5,j,k)}^{(0.5)} (\phi_{i+1,j,k} -\phi_{i,j,k})\\
& ~~~\quad  - A_{1,(i-.5,j,k)}^{(0.5)} (\phi_{i,j,k} -\phi_{i-1,j,k})\\
& ~~~\quad+  A_{2,(i,j+.5,k)}^{(0.5)} (\phi_{i,j+1,k} -\phi_{i,j,k})\\
& ~~~\quad  - A_{2,(i,j-.5,k)}^{(0.5)} (\phi_{i,j,k} -\phi_{i,j-1,k})\\
&~~~ \quad+  A_{3,(i,j,k+.5)}^{(0.5)} (\phi_{i,j,k+1} -\phi_{i,j,k})\\
&~~~ \quad  - A_{3,(i,j,k-.5)}^{(0.5)} (\phi_{i,j,k} -\phi_{i,j,k-1}) \Big).
\end{align*}
In these equations, $ds$ is the space discretization step and the notation for the discretized matrix entries means
\begin{equation*}
A_{m,(a,b,c)}^{(0.5)} = A_m(t+0.5h,[a-1,b-1,c-1]ds),~~~(j=1,2,3).
\end{equation*}
We use a $(N_1+1) \times (N_2+1) \times  (N_3+1)$ grid, corresponding to a spatial domain $\Omega \equiv [0,N_1ds] \times [0,N_2ds] \times [0, N_3ds]$.
The boundary condition $A(x,t) \nabla \phi(x,t)=0$ for $x \in \partial \Omega$ is implemented by requiring that $A_m(t,x)=0$ whenever $x \notin \Omega$. It may be verified that under these conditions, the following holds:
\begin{align}
\sum_{i=0}^{N_1} \sum_{j=0}^{N_2} \sum_{k=0}^{N_3} \text{div} \left(A^{(0.5)} \nabla \phi  \right)_{i,j,k} = 0,
\end{align}
in accordance with the divergence theorem.

\subsection{Simulation of spatially-averaged model and optimal pulse-only strategy}
We present first results from the spatially-averaged model, which are easier to display graphically. Following \cite{fotsa} we use an inhibition pressure of the form%
\begin{equation}
\label{EqInhibitionPressure}\alpha\left(  t\right)  =a\left(  t-b\right)
^{2}\left(  1-\cos\left(  2\pi t/c\right)  \right)  ,
\end{equation}
with $b, c \in\left[  0,1\right]  $. This function is shown in Figure~\ref{fig:inhib_press}; it reflects the seasonality of
empirically-based severity index models found in the literature
\cite{danneberger,dodd,duthie}.
\begin{figure}[th]
\begin{center}
  \includegraphics[width=0.5\textwidth]{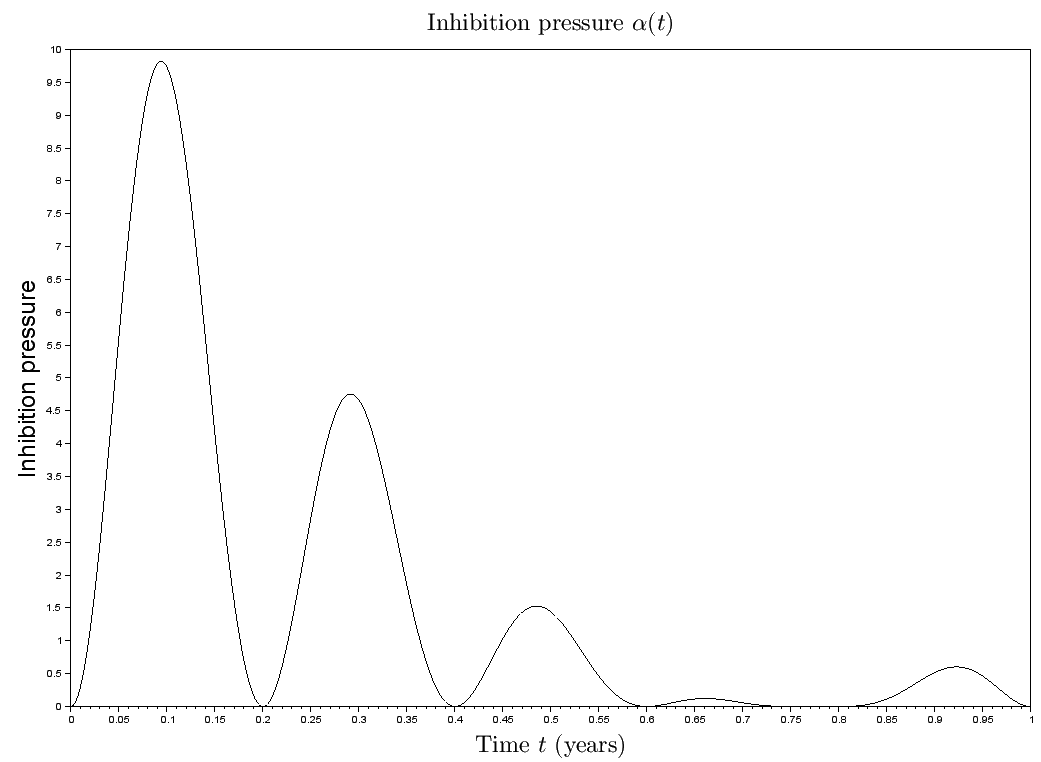}
  \caption{Inhibition pressure $\alpha(t)$}\label{fig:inhib_press}
\end{center}
\end{figure}

Parameters used in the simulation are summarized in Table
\ref{SimulationParamTable}. We considered only the case
where the cost of  pulse intervention is independent of time (that is, $c_{i}$ is
constant):  the
general features of the solution are still manifest in this special case.

\begin{center}
\begin{table}[th]%
\begin{tabular}
[c]{ccc}\hline\hline
Parameter & Significance & Value\\\hline
\multicolumn{1}{l}{$a $} & \multicolumn{1}{l}{Average amplitude
parameter in (\ref{EqInhibitionPressure})} & \multicolumn{1}{l}{$0.5\log\left(  10\right) $}\\\hline
\multicolumn{1}{l}{$b$} & \multicolumn{1}{l}{Time of maximal}
& \multicolumn{1}{l}{$0.75$}\\
\multicolumn{1}{l}{} & \multicolumn{1}{l}{inhibition pressure in (\ref{EqInhibitionPressure})} & \multicolumn{1}{l}{}\\\hline
\multicolumn{1}{l}{$c$} & \multicolumn{1}{l}{Period of oscillations in
(\ref{EqInhibitionPressure})} & \multicolumn{1}{l}{$0.2$}\\\hline
\multicolumn{1}{l}{$c_{i}$} & \multicolumn{1}{l}{Unit cost of pulse} &
\multicolumn{1}{l}{0.25, 0.4 \text{ or } 0.5}\\
\multicolumn{1}{l}{} & \multicolumn{1}{l}{intervention} & \multicolumn{1}{l}{}%
\\\hline
\multicolumn{1}{l}{$C_f$} & \multicolumn{1}{l}{Unit cost at harvest} &\multicolumn{1}{l}{0, 0.25, \text{ or } 0.5}
\\\hline
\multicolumn{1}{l}{$1-\sigma$} & \multicolumn{1}{l}{Inhibition rate attractor
if $u=1$} & \multicolumn{1}{l}{$0.7$}\\\hline
\multicolumn{1}{l}{$\sigma^{\ast}$} & \multicolumn{1}{l}{Pulse intervention
threshold} & \multicolumn{1}{l}{$0$}\\\hline
\multicolumn{1}{l}{$\Theta_{0}$} & \multicolumn{1}{l}{Average initial inhibition rate}
& \multicolumn{1}{l}{ $0.4$}\\\hline
\multicolumn{1}{l}{$\Delta\tau$} & \multicolumn{1}{l}{Time interval between} &
\multicolumn{1}{l}{$1/52$ (one week)}\\
\multicolumn{1}{l}{} & \multicolumn{1}{l}{pulse interventions} &
\multicolumn{1}{l}{}\\\hline
\multicolumn{1}{l}{$T$} & \multicolumn{1}{l}{Length of simulation} &
\multicolumn{1}{l}{$1$ (year)}%
\end{tabular}
\caption{Simulation parameters for spatially-averaged model}%
\label{SimulationParamTable}%
\end{table}
\end{center}

Simulation results for the averaged model are displayed in Figure~\ref{fig:pulse_strategy}, \ref{fig:chem_control}, and  \ref{fig:pulse_Cf}. In Figure~\ref{fig:pulse_strategy}, the optimal  pulse-only control strategy is shown for a range of pulse control costs. (The optimal pulse control strategy was calculated based on  Theorem \ref{eq:OptStrategyAverage2}, with $p_v$ calculated as  the numerical solution to the adjoint problem (\ref{EqAdjp1})--(\ref{EqAdjp2}).  As expected, as intervention cost increases the amount of intervention decreases for the optimal solution. It is interesting to note in Figure~\ref{fig:pulse_strategy} that the sets of intervention times corresponding to successively larger intervention costs form a sequence of nested subsets. However, we have not proven that this is true in general.
\begin{figure}[th]
\begin{center}
  \includegraphics[width=0.50\textwidth]{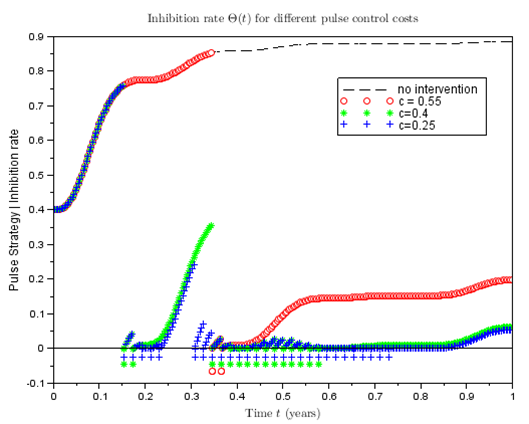}
  \caption{Inhibition rate with optimal pulse-only control for different unit pulse control costs. The markings below the time axis indicate times at which pulse control is applied for the corresponding control cost. The unit cost at harvest ($C_f$) was set equal to 0.}\label{fig:pulse_strategy}
\end{center}
\end{figure}
Figure~\ref{fig:chem_control} shows optimal pulse control for the same unit pulse control costs, but with a constant chemical control $u=1$ and inhibition rate attractor $1 - \sigma = 0.7$ .  The maximum inhibition rate is reduced by the chemical control, and the number of pulse interventions is also reduced compared to Figure~\ref{fig:pulse_strategy}.
\begin{figure}[th]
\begin{center}
 \includegraphics[width=0.50\textwidth]{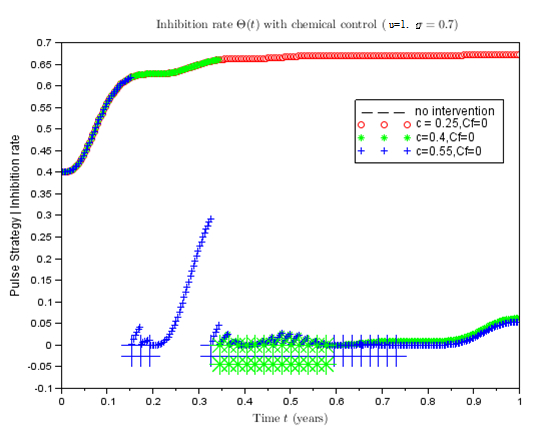}
  \caption{Inhibition rate with optimal  pulse control for different unit pulse control costs, given that constant chemical control is also used $(u=1, \sigma = 0.3)$. The markings below the time axis indicate times at which pulse control is applied for the corresponding control cost. The unit cost at harvest ($C_f$) was set equal to 0.}\label{fig:chem_control}
\end{center}
\end{figure}
Figure~\ref{fig:pulse_Cf} illustrates the effect of unit cost at harvest ($C_f$) on the optimal strategy. In this case, increases in final control cost lead to increases in the application of pulse control that effectively reduce the final cost. As in Figure~\ref{fig:pulse_strategy}, the pulse control applications for different final control costs form a sequence of nested subsets. Note that if $C_f > c$ then $C_f$ has no further influence on the optimal strategy, because then it is always preferable to use pulse control immediately before harvest rather than to incur the harvest cost $C_f$.
\begin{figure}[th]
\begin{center}
  \includegraphics[width=0.50\textwidth]{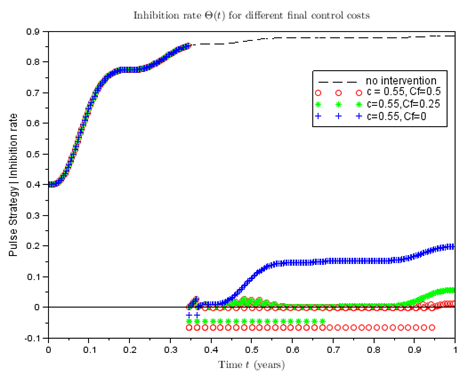}
  \caption{Inhibition rate with optimal pulse-only control for different unit final control costs. The markings below the time axis indicate times at which pulse control is applied for the corresponding final control cost.}\label{fig:pulse_Cf}
\end{center}
\end{figure}

\subsection{Simulation of main  model  and optimal pulse-only strategy}
We also simulated the main model, using the parameters specified in
Tables~\ref{SimulationParamTable} and \ref{General SimulationParamTable}. 
We first considered the case where   the inhibition pressure is independent of location,
while the initial inhibition rate varies with location according to the functional form
\begin{equation}
q_1 \left( \sin\left(\frac{\pi x_1}{N_1ds}\right) \sin\left(\frac{\pi x_2}{N_2ds}\right) \sin\left(\frac{\pi x_3}{N_3ds}\right) \right)^{1/3} + q_2.
\end{equation}
This function indicates an initial infection that is concentrated towards the center of $\Omega$.  The constants $q_1$ and $q_2$ were chosen such that the average initial inhibition rate agrees with Table~\ref{SimulationParamTable}. 

\begin{center}
\begin{table}[th]%
\begin{tabular}
[c]{ccc}\hline\hline
Parameter & Significance & Value\\\hline
\multicolumn{1}{l}{$A$} & \multicolumn{1}{l}{Diffusion matrix} &
\multicolumn{1}{l}{$I, 10I$}\\\hline
\multicolumn{1}{l}{$\Omega$} & \multicolumn{1}{l}{grid size ($N_1\times N_2\times
N_3$)} & \multicolumn{1}{l}{$10\times 10 \times 3$}\\\hline
\multicolumn{1}{l}{$ds$} & \multicolumn{1}{l}{spatial grid spacing} & \multicolumn{1}{l}{$1$}\\\hline
\multicolumn{1}{l}{} & \multicolumn{1}{l}{} & \multicolumn{1}{l}{}%
\end{tabular}
\caption{Additional simulation parameters for general model}%
\label{General SimulationParamTable}%
\end{table}
\end{center}

Figures~\ref{fig:MainControl1} and~\ref{fig:MainControl2} show the evolution under optimal pulse strategy for all grid points in a $10 \times 10 \times 3$ grid. The systems represented by the two figures have different diffusion matrices ($A=I$ for Figure~\ref{fig:MainControl1}, $A = 10I$ for Figure~\ref{fig:MainControl2}). The figures confirm that when the inhibition rate and costs are spatially indepent,  then the optimal strategy is also spatially independent:  pulse interventions are always applied to the entire region, and never to a proper subregion. 
Furthermore, the optimal strategy does not depend on the initial inhibition rate distribution, or on the diffusion matrix $A$. All of these characteristics may be rigorously proven using the fact that the optimal strategy is derived from the solution of the adjoint system (\ref{EqAdjp1PDE})--(\ref{EqAdjp3PDE}), which is independent of
$\theta_0$ and is also independent of  $x$ as long as $C_f$ and $c$ are independent of $x$. Since the optimal strategy is space-independent, we find that in this case the spatial average of the general model agrees exactly with the averaged model. This illustrates the practical usefulness of the averaged model, in the case where inhibition pressure and costs are independent of spatial location.
\begin{figure}[th]
\begin{center}
  \includegraphics[width=0.50\textwidth]{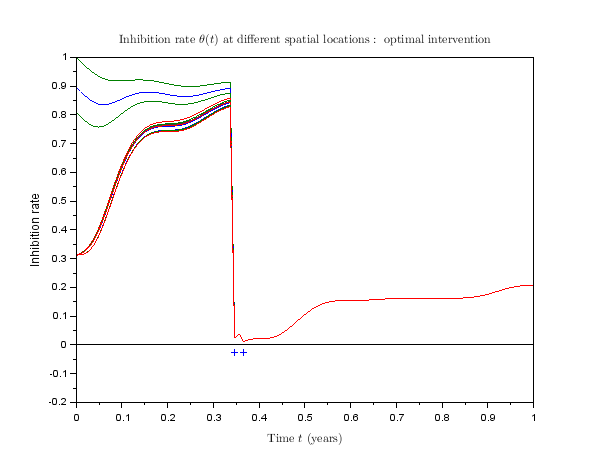}
  \caption{Inhibition rate with optimal pulse control as a function of time for different grid points on a $10\times 10 \times 3$ grid, for a system with $A  = I, c = 0.55, C_f=0,$ and other parameters as given in Tables~\ref{SimulationParamTable} and \ref{General SimulationParamTable}. The space-averaged version of this system is associated with the $c=0.55$ curve in Figure~\ref{fig:pulse_strategy}.  Note the optimal pulse interventions (indicated by the markings below the time axis) are the same in both cases. }\label{fig:MainControl1}
\end{center}
\end{figure}
\begin{figure}[th]
\begin{center}
  \includegraphics[width=0.50\textwidth]{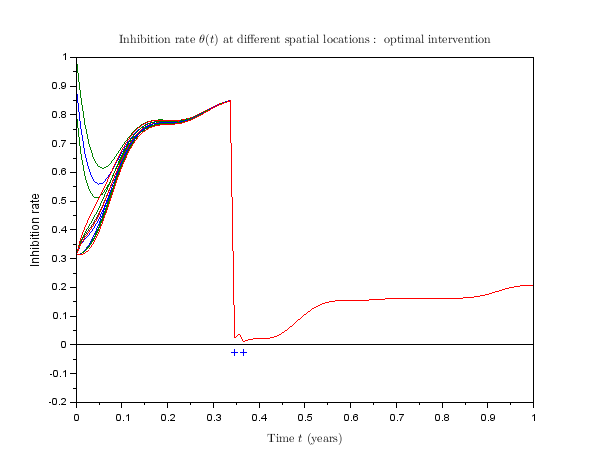}
  \caption{Inhibition rate with optimal pulse control  for different grid points, for a system as in Figure~\ref{fig:MainControl1} except with diffusion matrix $A=10I$.}\label{fig:MainControl2}
\end{center}
\end{figure}

Next, we consider the case where the inhibition pressure depends on spatial location. 
We did this by choosing  $a(x)$ in  (\ref{EqInhibitionPressure}) for each grid point independently according to a uniform random distribution, then rescaling so that the average value of $a(x)$ agrees with Table~\ref{SimulationParamTable}. Other parameters from Table~\ref{SimulationParamTable} remain unchanged, and  are constant with respect to the space variable. In particular, the initial
conditions $\rho(x)$ were taken as constant with respect to the space variable. Figure \ref{fig:MainControl3} shows the inhibition rate as a function of time for all grid points of a $10 \times 10 \times 3$ grid when optimal pulse control is applied. Unlike the cases shown in Figures~\ref{fig:MainControl1} and ~\ref{fig:MainControl2}, in this case the optimal strategy depends on spatial position. 
\begin{figure}[th]
\begin{center}
  \includegraphics[width=0.50\textwidth]{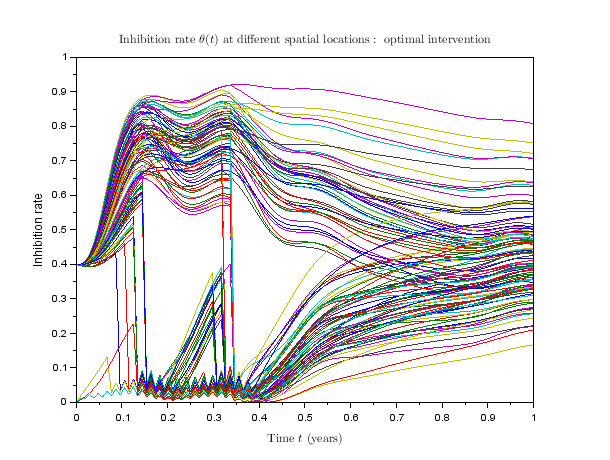}
  \caption{Inhibition rate with optimal pulse control for a system as in Figure~\ref{fig:MainControl1} except with constant initial inhibition pressure, and with space-dependent inhibition rate.}\label{fig:MainControl3}
\end{center}
\end{figure}

%
%
%
\section{Discussion\label{Discussion}}
This paper significantly extends the results of \cite{fotsa} in several respects. 
First, the model has been generalized by imposing weaker smoothness conditions on the parameters.
Parameters are only required to be measurable and
essentially locally bounded on $%
\mathbb{R}
^{3+1}$. Under these conditions, we have showed existence and uniqueness of a solution which takes values in $\left[  0,1\right]  $, thus establishing the model to be well formulated
mathematically and epidemiologically.
Second, we have included the possibility of a pulse control strategy, along with the continuous (chemical) control strategy studied in \cite{fotsa}.  The added pulse strategy
$v=\left(  v_{i}\right)  _{i\in%
\mathbb{N}
}$ represents the cultivational practices such as pruning old infected twigs,
removing mummified fruits \cite{boisson,mouen072,mouen03,mouen08,wharton}.
Third, we have verified an explicit algorithm for finding the optimal pulse-only strategy in the case where $\sigma^{\ast} = 0$.
Numerical simulations  for both the averaged version of
the model and the full version confirm the practical applicability of this algorithm. As explained in subsection
\ref{SubsectionIDEModel}, the averaged version of the model faithfully represents the
average behavior on the bounded domain $\Omega$, when inhibition pressure and intervention costs are space-independent.


Practical computation of control strategies that optimize the use of  both pulse and chemical contr are the subject of ongoing research.

\begin{center}
\textbf{ACKNOWLEDGEMENTS}
\end{center}

The author thank Doctor Martial NDEFFO MBAH from the Yale School of Medecine
and the University of Cambridge for all his useful remarks concerning the
presentation of the paper. They also thank the African Mathematics Millennium Science Initiative (AMMSI) for the financial support they gave to the first author in the year 2014.

\end{document}